\documentclass{jnsao}
\usepackage[utf8]{inputenc}
\usepackage[finnish,english]{babel}
\usepackage[svgnames]{xcolor}
\usepackage{hyperref}
\usepackage{tikz, pgfplots}
\usepgfplotslibrary{colorbrewer}
\usepackage[inline]{enumitem}
\usepackage{algpseudocode}
\usepackage[nameinlink,capitalise]{cleveref}
\usepackage{nicematrix}
\usepackage{booktabs}
\usepackage{pgfplotstable}

\crefformat{equation}{\upshape{(#2#1#3)}}
\crefrangeformat{equation}{\upshape{(#3#1#4)--(#5#2#6)}}
\crefmultiformat{equation}{\upshape{(#2#1#3)}}%
{ and~\upshape{(#2#1#3)}}{, \upshape{(#2#1#3)}}{ and~\upshape{(#2#1#3)}}
\crefformat{enumi}{\upshape{#2#1#3}}
\crefrangeformat{enumi}{\upshape{#3#1#4--#5#2#6}}
\crefmultiformat{enumi}{\upshape{#2#1#3}}%
{ and~\upshape{#2#1#3}}{, \upshape{#2#1#3}}{ and~\upshape{#2#1#3}}

\makeatletter
\let\@the@H@page\relax
\makeatother

\definecolor{hrefcolor}{rgb}{0.0,0.4,0.7}
\definecolor{citecolor}{rgb}{0.0,0.35,0.2}
\definecolor{structure}{rgb}{0.09,0.09,0.44}
\hypersetup{
    colorlinks=true,
    linkcolor=citecolor,
    citecolor=citecolor,
    filecolor=hrefcolor,
    urlcolor=hrefcolor,
    pdfencoding=auto,
    hypertexnames=false,
}

\manuscriptcopyright{}
\manuscriptlicense{}
\manuscripteprint{2606.10951}
\manuscripteprinttype{arxiv}

\usepackage{float}

\usepackage{xcolor,cancel}

\floatstyle{ruled}
\floatname{algorithm}{Algorithm}
\newfloat{algorithm}{tbp!}{loa}
\numberwithin{algorithm}{section}

\tikzset{notestyleraw/.append style={align=justify}}

\usepackage{tikz}
\usepackage{pgfplots}
\usepgfplotslibrary{colorbrewer}

\def\ignorelegendentry#1{}
\pgfplotsset{
    every axis label/.append style = {font = \scriptsize},
    every tick label/.append style = {font = \scriptsize},
    ignore legend/.style={
        every axis legend/.code={\let\addlegendentry\ignorelegendentry}
    },
    log x ticks with fixed point/.style={
        xticklabel={
            \pgfkeys{/pgf/fpu=true}
            \pgfmathparse{exp(\tick)}%
            \pgfmathprintnumber[fixed relative, precision=3]{\pgfmathresult}
            \pgfkeys{/pgf/fpu=false}
        },
    },
    log y ticks with fixed point/.style={
        yticklabel={
            \pgfkeys{/pgf/fpu=true}
            \pgfmathparse{exp(\tick)}%
            \pgfmathprintnumber[fixed relative, precision=3]{\pgfmathresult}
            \pgfkeys{/pgf/fpu=false}
        },
    },
    compat=1.18,
    every axis/.append style={
        outer sep = 0pt,
        title style = {
            font = \normalsize
        },
        every x tick label/.append style={
          /pgf/number format/1000 sep={\ },
        },
        every y tick label/.append style={
          /pgf/number format/1000 sep={\ },
        }
    }
}

\usetikzlibrary{external}
\tikzset{external/optimize=true}
\tikzexternaldisable

\usepackage{mdframed}

\mdfdefinestyle{examplebg}{
    backgroundcolor=black!6,%
    hidealllines=true,%
    innertopmargin=.3em,%
    innerbottommargin=.7em,%
    innerleftmargin=.7em,%
    innerrightmargin=.7em,%
    footnoteinside=false,
}

\surroundwithmdframed[style=examplebg]{example}

\author{%
    Jyrki Jauhiainen\thanks{%
        Department of Mathematics and Scientific Computing, University of Graz, Austria.
        \mbox{\email{jyrki.jauhiainen@uni-graz.at}},
        \orcid{0000-0001-6711-6997}
    }
    \and
    Yassine Nabou\thanks{%
        Department of Mathematics and Statistics, University of Helsinki, Finland.
        \mbox{\email{yassine.nabou@helsinki.fi}},
        \orcid{0009-0004-9805-8039}
    }
    \and
    Tuomo Valkonen\thanks{%
        MODEMAT Research Center in Mathematical Modeling and Optimization, Quito, Ecuador
        \emph{and}
        Department of Mathematics and Statistics, University of Helsinki, Finland.
        \email{tuomo.valkonen@iki.fi},
        \orcid{0000-0001-6683-3572}
    }
}

\shortauthor{J.~Jauhiainen, Y.~Nabou, T.~Valkonen}

\acknowledgements{%
    This research has been supported by the Academy of Finland grant 345486.
    This research was funded in whole or in part by the Austrian Science
    Fund (FWF) \href{https://dx.doi.org/10.55776/F100800}{10.55776/F100800}.
}

\title{Dynamic inverse problems: Single-loop online algorithms}

\newcommand{\term}{\emph}

\newcommand{\field}[1]{\mathbb{#1}}
\newcommand{\N}{\mathbb{N}}

\newcommand{\R}{\field{R}}
\newcommand{\extR}{\overline \R}

\newcommand{\norm}[1]{\|#1\|}

\newcommand{\abs}[1]{|#1|}

\newcommand{\inv}[1]{#1^{-1}}

\newcommand{\grad}{\nabla}
\newcommand{\freevar}{\,\boldsymbol\cdot\,}
\newcommand{\Union}\bigcup
\newcommand{\Isect}\bigcap
\newcommand{\union}\cup
\newcommand{\isect}\cap
\newcommand{\bigunion}\bigcup
\newcommand{\bigisect}\bigcap

\newcommand{\defeq}{:=}

\newcommand{\subdiff}{\partial}

\DeclareRobustCommand{\downto}{{{\mathchoice%
                    {\rotatebox[origin=c]{-20}{$\to$}}
                    {\rotatebox[origin=c]{-20}{$\to$}}
                    {\rotatebox[origin=c]{-20}{\scalebox{0.75}{$\to$}}}
                    {\rotatebox[origin=c]{-20}{\scalebox{0.6}{$\to$}}}
                }}}
\DeclareRobustCommand{\upto}{{{\mathchoice%
                    {\rotatebox[origin=c]{20}{$\to$}}
                    {\rotatebox[origin=c]{20}{$\to$}}
                    {\rotatebox[origin=c]{20}{\scalebox{0.75}{$\to$}}}
                    {\rotatebox[origin=c]{20}{\scalebox{0.6}{$\to$}}}
                }}}

\DeclareMathOperator{\closure}{cl}

\DeclareMathOperator{\conv}{conv}

\DeclareMathOperator{\diag}{diag}

\DeclareMathOperator{\dist}{dist}
\DeclareMathOperator{\prox}{prox}

\newcommand{\iprod}[2]{\langle #1,#2\rangle}

\def\WOp{\Sigma^{-1/2}}
\def\EITmeas{\mathscr{I}}

\def\PDpredictConstr{\mathscr{U}}

\def\ProductSpace{U}
\def\primaldynamics{\bar{\mathscr{X}}}
\def\PpredictConstr{\mathscr{X}}
\def\DpredictConstr{\mathscr{Y}}
\def\PDpredictConstr{\mathscr{U}}

\newcommand{\Precond}{M}

\newcommand{\EkGrowthMono}[1][k]{{\hat\gamma}_{E}}
\newcommand{\EkLossMono}[1][k]{{\hat\lambda}_{E}}
\newcommand{\EkLossGlobal}[1][k]{{\bar\lambda}_{E}}
\newcommand{\EkGrowthGlobal}[1][k]{{\bar\gamma}_{E}}
\newcommand{\EkLoss}[1][k]{\lambda_{E}}
\newcommand{\EkGrowth}[1][k]{\gamma_{E}}
\def\Err{e}
\def\ErrMono{\hat e}
\def\ErrGlobal{\bar e}
\def\GrowthOpMono{\hat\Gamma}

\def\GrowthOp{\Gamma}

\def\criticalprox{\hat r}

\def\YoungCoeff{\xi}
\def\deltaCoeff{\theta}
\def\Perturbr{\Delta}
\def\EkLipCoeffCoeff{\kappa}

\def \weaktostarSym{\setbox0=\hbox{$\rightharpoonup$}\rlap{\hbox
        to\wd0{\hss\raise1ex\hbox{$\scriptscriptstyle{*\,}$}\hss}}\box0}

\def\linear{\mathbb{L}}
\newcommand{\setto}{\rightrightarrows}

\def\extR{\overline \R}

\def\dualprod#1#2{\langle #1|#2\rangle}
\newcommand{\adaptdualprod}[2]{\left\langle #1 \middle| #2\right\rangle}

\def\infconv{\mathop{\Box}}

\let\phi=\varphi
\let\epsilon=\varepsilon

\DeclareMathOperator{\Id}{Id}

\def\nexxt#1{#1^{k+1}}
\def\this#1{#1^k}
\def\prev#1{#1^{k-1}}
\let\opt\hat

\def\estgrad#1#2{\widetilde{\grad #1}_{#2}}
\def\estdiff#1#2{\widetilde{#1'_{\phantom{k\,}}}\!\!\!\!{\lower2pt\hbox{$_{#2}$}}}

\def\nextx{\nexxt x}
\def\thisx{\this x}
\def\prevx{\prev{x}}
\def\optx{\opt x}

\def\nexty{\nexxt y}
\def\thisy{\this y}

\def\nextu{\nexxt u}
\def\thisu{\this u}
\def\optu{\opt u}

\def\nextdiff{E_{{k+1}}'(\predicted{x}{k+1})}
\def\nextestdiff{\estdiff{E}{{k+1}}(\predicted{x}{k+1})}
\def\nextestgrad{\estgrad{E}{{k+1}}(\predicted{x}{k+1})}
\def\thisestdiff{\estdiff{E}{{k}}(\predicted{x}{k})}
\def\thisestgrad{\estgrad{E}{{k}}(\predicted{x}{k})}
\def\thisgrad{\grad E_k(\predicted{x}{k})}
\def\nextgrad{\grad E_k(\predicted{x}{k+1})}
\def\thisdiff{E_k'(\predicted{x}{k})}

\def\diffwrt#1#2{#1^{(#2)}}

\def \approxinSym{\setbox0=\hbox{$\in$}\rlap{\hbox
        to\wd0{\hss\raise1ex\hbox{$\sim$}\hss}}\box0}

\def\localset{\Omega}

\usepackage{ifthen}

\newcommand{\dis}[4][]{%
    d_{#2}%
    \ifthenelse{\equal{#1}{2}}{^2}{}%
    (#3, #4)%
}

\newcommand{\adaptdis}[4][]{%
    d_{#2}%
    \ifthenelse{\equal{#1}{2}}{^2}{}%
    \left(#3, #4\right)%
}

\newcommand{\fakenorm}[1]{\llbracket #1 \rrbracket}

\def\primaldifffact{\mu_u}

\def\trackingres{\psi}

\newcommand{\trackingressum}[1][p]{\varsigma_{#1}}

\def\innerError#1{\varepsilon_{u,#1}}
\def\adjointError#1{\varepsilon_{w,#1}}
\def\transformError#1{{\tilde\varepsilon}_{#1}}

\def\predicted#1#2{\breve{#1}^{#2}}
\def\realopt#1{\hat #1}
\def\realoptu{{\realopt{u}}}

\def\realoptx{{\realopt{x}}}

\def\realopty{{\realopt{y}}}

\def\primalpredict{\breve x}

\def\pdpredict{\breve u}
\def\overnext#1{\bar #1^{k+1}}
\def\overnextu{\overnext{u}}
\let\MAX\overline

\def\combinederror{\mathring e}

\def\Nelec{N_{\text{elec}}}
\def\Nmeas{N_{\text{inj}}}

\begin{document}

\maketitle

\begin{abstract}
    We study efficient online methods for dynamic inverse problems with infinite time horizon.
    We concentrate, in particular, on problems whose forward model arises from a PDE.
    Our motivating application is flow monitoring with Electrical Impedance Tomography (EIT).
    The idea of such online methods is to take single steps of of standard optimisation algorithms, on each time index; each data frame.
    A predictor, based on problem dynamics, is used to transfer iterates one from time index to the next one.
    If we monitor a fast flow with a correspondingly fast measurement modality, such as EIT, basic methods are unable to solve the PDE before new data arrives.
    Our idea, then, is to not solve it, and instead, on each iteration, each time index, take single or few steps of standard iterative solvers towards the solution of both the PDE and an adjoint PDE.
    This is what “single loop” refers to.
    To the overall problem, we apply standard online optimisation methods, at the outside developed for exact gradients $\grad E_k(x^k)$ of the iteration-dependent data fidelity $E_k$ that incorporates the PDE.
    We replace the gradient by a single-loop estimate $\widetilde{\grad E_k}(x^k)$ that satisfies standard smoothness properties with summable errors.
    This allows standard regret proofs to go through.
    Our numerical experiments on dynamic EIT validate the theoretical predictions and highlight the potential of the proposed approach for the real-time solution of PDE-constrained dynamic inverse problems.
\end{abstract}

\section{Introduction and Motivation}
\label{sec:intro}

Electrical Impedance Tomography (EIT) is a non-invasive imaging modality in which the internal electrical conductivity of a body is reconstructed from boundary voltage--current measurements. Due to its low cost, portability, and absence of ionising radiation, EIT has found applications in medical imaging (e.g.\ lung ventilation monitoring), geophysics, and industrial process monitoring. In many of these applications, the conductivity distribution evolves over time, for instance due to physiological dynamics, transport processes, or changing environmental conditions. As a consequence, measurements are acquired sequentially and reconstructions must be updated in real time, giving rise to \emph{online} or \emph{dynamic} inverse problems \cite{ip_special_issue_2028,holland2010reducing,hunt2014weighing,lipponen2011nonstationary,tuomov2024online-eit,alsaker2023multithreaded,holder2026electrical}.

If we wish to solve the dynamic EIT problem, or other challenging dynamic inverse problem, in real-time, as the data progresses, we need very fast algorithms. If we want a high resolution solution, we may not have time to wait for partial differential equations (PDE) governing the problem to be solved; new data will arrive faster.
In fact, we will usually also need to solve an adjoint PDE, adding to the computational cost.
In this work, based on the idea of \emph{single-loop online differential estimation}, we will develop algorithms that are able to provide real-time solutions---by never solving the PDEs.

\paragraph{Mathematical model of EIT}

At time step $k$, the forward problem of dynamic EIT is governed by the Complete Electrode Model (CEM). Given a conductivity distribution $x^k: \Omega \to \R$, the electric potential $u^{j,k}: \Omega \to R$ corresponding to the $j$th current injection ($j=1,\ldots,\Nmeas$) satisfies
\begin{subequations}
    \label{eq:eit:cem}
    \begin{align}
        \nabla\cdot(x^k\nabla u^{j,k})                           & = 0          &  & \text{in }\Omega,
        \\
        u^{j,k} + \zeta_i x^k\nabla u^{j,k}\cdot\nu              & = U_i^{j,k}  &  & \text{on }\partial\Omega_{e_i},
        \\
        x^k\nabla u^{j,k}\cdot\nu                                & = 0          &  & \text{on }\partial\Omega\setminus\cup_i \partial\Omega_{e_i},
        \\
        \int_{\partial\Omega_{e_i}}x^k\nabla u^{j,k}\cdot\nu\,dS & = -I_i^{j,k} &  & \text{for }i=1,\dots,\Nelec.
    \end{align}
\end{subequations}
Here $\Omega\subset\mathbb{R}^d$ denotes the imaging domain, $\partial\Omega_{e_i}$ are the electrode surfaces, $\zeta_i>0$ are contact impedances, and $\nu$ is the outward unit normal.
The electrode potentials $U_i^{j,k}$ for electrodes $i=1,\ldots,\Nelec$ are prescribed, while the resulting electrode currents $I_i^{j,k}$ are determined by the model.
We adopt a potential-to-current formulation of the CEM, in which the electrode potentials $U^{j,k}$ are prescribed and the corresponding electrode currents $I^{j,k} = I(x^k, U^{j,k})$ are computed. This reflects modern EIT measurement devices \cite{jauhiainen2020relaxed}.
Here $\EITmeas^{j,k}$ denotes noisy measurements of the electrode currents corresponding to the applied electrode potentials $U^{j,k}$.
The inverse problem consists in recovering $x^k$ from noisy boundary measurements $\EITmeas^{j,k}$.
Applying total variation regularisation, this is commonly formulated the optimisation problem
\[
    \min_x E_k(x) + F_k(x) + G_k(K_k x)
\]
for $K_k$ a (discretised) gradient operator,
\begin{equation}
    \label{eq:eit:functionals}
    E_k(x) \defeq
    \frac12\sum_{j=1}^{\Nmeas}
    \norm{\WOp ( I(x,U^{j,k}) - \EITmeas^{j,k})}_2^2,
    \quad
    F_k(x) \defeq \delta_{[x_m,x_M]}(x),
    \quad\text{and}\quad
    G_k(y) \defeq \alpha \norm{y}_{2,1},
\end{equation}
where $0<x_m<x_M$ enforce physical bounds, $\WOp$ is a data precision matrix, and $\alpha>0$ a regularisation parameter.
Similar formulations have been studied in \cite{voss2018imaging,voss2019three,jauhiainen2021nonplanar}.

\paragraph{General online optimisation framework}

Motivated by online EIT and related applications, we consider the general infinite time horizon optimisation problem
\begin{equation}
    \label{eq:online:problem}
    \min_{(x^0,x^1,\dots)\in\primaldynamics} \sum_{k=0}^{\infty} V_k(x^k),
    \quad\text{for}\quad
    V_k(x) \defeq E_k(x) + G_k(K_k x) + F_k(x)
    \quad\text{and}\quad
    \primaldynamics \subset \prod_{k=0}^\infty X_k,
\end{equation}
where $E_k: X_k \to \mathbb{R}$ is smooth but possibly nonconvex, $K_k \in \mathbb{L}(X_k;Y_k)$, $G_k: Y_k \to \mathbb{R}$ and $F_k: X_k \to \mathbb{R}$ are convex and possibly nonsmooth. $X_k,Y_k$ are normed spaces.
The set $\primaldynamics$ encodes the dynamics of the problem.
The data fidelity term is of the composite form
\begin{equation}
    \label{eq:tracking:implicit-inner}
    E_k(x)=J_k(S_{k,u}(x)),
    \qquad
    0 = T_k(S_{k,u}(x),x),
\end{equation}
and $S_{k,u}: X_k \to U_K$ denotes the solution operator for the PDE constraint $T_k(u,x)=0$.
Here, the subscript $u$ is symbolic, indicating that the operator maps into the space $U_k$ (to distinguish it from the adjoint solution map $S_{k,w}$ defined later).
That is, for each $x$, the state $u=S_{k,u}(x) \in U_k$ satisfies the underlying (possibly nonlinear) PDE.
This formulation encompasses EIT as well as a broad class of dynamic inverse and control problems.

One challenge with this type of problems is that $S_{k,u}(x^k)$ as well as the derivative $S_{k,u}'(x^k)$, required by many optimisation algorithms, can be numerically expensive to evaluate.
Consider, in particular $S_{k,u}(x^k)=((I^{j,k}, u^{j,k}))_{j=1}^{\Nmeas}$ for EIT.
We cannot, in general, prefactorise the mass matrix of \eqref{eq:eit:cem}, because of the changing conductivity. But factorising a matrix on each iteration is expensive, while methods that do not require factorisation, can be slow to converge.
Therefore, in \cite{jensen2022nonsmooth,tuomov2024tracking}, still in the context of static (non-dynamic, non-online) problems, we introduced provably convergent “single-loop” approaches that avoid solving the PDE, by only taking, on each outer iteration, a single or few iterations of conventional splitting methods towards the solution.

\paragraph{Contributions}

Our goal is to extend those approaches to the dynamic and online setting for the problem~\eqref{eq:online:problem}. To do this,
in \cref{sec:online-tracking,sec:inner-adjoint}, we develop \textbf{single-loop estimates $\thisestdiff$ of differentials} $\thisdiff$ of composite functions $E_k = J_k \circ S_{k,u}$, where $S_{k,u}$ denotes the solution mapping of the inner PDE constraint.
The idea is that on each data frame or outer iteration $k$, we only make few very simple operations to update our estimate $\thisestdiff$ into $\nextestdiff$.
Here $\predicted{x}{k}$ is a prediction of $\thisx$, based on the dynamics of the problem.
The latter is formed by a correction step performed by the optimisation method.
This prediction is essential, as, in a dynamic problem, $\prevx$ may no longer be valid or good starting point for the optimisation problem, the data of the problem having changed.
This fact is also modelled by the time-indexed spaces $X_k$.

The differential estimates lift standard iterative algorithms, such as Gauss–Seidel splitting, for estimating $S_{k,u}(\predicted{x}{k})$, as well as an adjoint equation that defines $\thisdiff=J_k'(S_{k,u}(\predicted{x}{k}))S_{k,u}'(\predicted{x}{k})$.
However, on a single outer iteration or data frame $k$, we never iterate these \emph{inner and adjoint algorithms} to completion: our goal is to---justifiably---only take a single step or a small number of step of these algorithms.
This is what \emph{single-loop} refers to.

The core idea, rigorously developed in \cref{sec:online-tracking}, is that all the complexity of the analysis single-loop methods reduces to the estimates $\thisestdiff$ of $\thisdiff$ satisfying standard smoothness properties with controlled error.
Rather than recomputing inner and adjoint solutions from scratch on each time step, the key to this working, is to recursively exploit already computed information from previous data frames, and to track and control the propagation of errors.
In \cref{sec:inner-adjoint} we show that the steps of standard splitting algorithms are capable of this, when combined with appropriate \term{predictors} to account for the dynamic nature of the problem.

These new online differential estimates can directly be used in the (outer) online primal-dual method of \cite{tuomov2024online-eit}, which we recall in \cref{sec:opdps}.
That method is a further development of the convex optimisation approach of \cite{tuomov-predict,tuomov-better-predict}, applicable to linear inverse problems such as magnetic resonance imaging (MRI) or positron emission tomography (PET), to nonlinear inverse problems and corresponding nonconvex optimisation problems.
We also previously applied the method to dynamic EIT in \cite{tuomov2024online-eit}, however using intermittent PDE solutions from a background thread to obtain real-time performance.
We demonstrate numerically in \cref{sec:numerical} that the new single-loop differential estimates can provide much more stable solutions without thread synchronisation difficulties and---as follows from \cref{sec:online-tracking,sec:opdps}---with less analytical complexity.

Our above-stated approach to generating the differential estimates extends to the online and dynamic setting, the  theories of \cite{jensen2022nonsmooth,tuomov2024tracking,suonpera2024general} regarding the single-loop solution of PDE constraints, inner optimisation problems, and their adjoint problems, of static (non-online) optimisation problems.
Alternative real-time algorithms for dynamic EIT include the multi-threaded approach of \cite{alsaker2023multithreaded}, based on the D-bar method.
Our companion paper \cite{online-regtheory} studies relevant dynamic regularisation theory: how do the algorithmic solutions behave as the noise level $\delta \downto 0$, the corresponding regularisation parameter $\alpha_\delta \downto 0$, and the frame $k \upto \infty$.
In the present paper, which treats the optimisation aspects, the noise level $\delta$ is constant.

\paragraph{Notation and elementary results}

We denote the extended reals by $\extR \defeq [-\infty,\infty]$.
We write $H: X \setto Y$ when $H$ is a set-valued map from the space $X$ to $Y$.
We write $\mathbb{L}(X;Y)$ for the space of bounded linear operators between the normed spaces $X$ and $Y$.
For Fréchet differentiable $F:X \to R$, we write $F'(x) \in X^*$ for the Fréchet derivative at $x \in X$. Here $X^*$ is the dual space to $X$.
For a convex function $F: X \to \extR$, we write $\subdiff F: X \setto X^*$ for its subdifferential map.
On a normed space $X$, for a point $x \in X$ and a set $U \subset X$, we write $\dist(x, U) \defeq \inf_{x' \in U} \norm{x-x'}_X$, where $\norm{\freevar}_X$ is the norm on $X$. We also write $\dist^2(x, U) \defeq \dist(x, U)^2$.
We write $\iprod{x}{x'}$ for the inner product between two elements $x$ and $x'$ of a Hilbert space $X$, and $\dualprod{x^*}{x} \defeq x^*(x)$ for the dual product or dual pairing in a Banach space.
We only work with real Hilbert spaces.
We write $\Id: X \to X$ for the identity operator on $X$ and $\delta_A: X \to \extR$ for the $\{0,\infty\}$-valued indicator function of a set $A \subset X$.

For $X$ a Hilbert space, we will frequently use Pythagoras' three-point identity
\begin{equation}
    \label{eq:intro:three-point}
    \iprod{x-y}{x-z}_X = \frac{1}{2}\norm{x-y}_X^2 - \frac{1}{2}\norm{y-z}_X^2 + \frac{1}{2}\norm{x-z}_X^2
    \quad
    (x,y,z\in X)
\end{equation}
and (inner product) Young's inequality
\begin{equation}
    \label{eq:intro:young}
    \iprod{x}{y}
    \le
    \norm{x}_X\norm{y}_X \le
    \frac{1}{2\alpha}\norm{x}_X^2 + \frac{\alpha}{2}\norm{y}_X^2
    \quad(x,y \in X,\, \alpha>0).
\end{equation}

Finally, given a set $\mathscr{X} \subset \prod_{k=0}^\infty X_k$, we define for $n \le m$ the slices $\mathscr{X}_{n:m} \defeq \{x_{n:m} \mid, (x_0, x_1, \ldots) \in \mathscr{X} \}$, where $x_{n:m} \defeq (x_n, \ldots x_m)$. We may use superscripts instead of subscripts in the case of iterates of algorithms.

\section{An inexact online primal-dual method}
\label{sec:opdps}

Recall the infinite-horizon optimisation problem \eqref{eq:online:problem},
\begin{equation}
    \label{eq:pd:infinite-horizon-problem}
    \min_{(x^0,x^1,\dots)\in\primaldynamics} \sum_{k=0}^{\infty} V_k(x^k),
    \quad\text{for}\quad
    V_k(x) \defeq E_k(x) + G_k(K_k x) + F_k(x),
\end{equation}
where $E_k$ is a possibly nonconvex but finite-valued function, $F_k$ and $G_k$ are convex, proper, lower semicontinuous functions, and $K_k$ is a linear operator.
The set $\primaldynamics \subset \prod_{k=0}^\infty X_k$ encodes the primal dynamics.
Note that, due to the Hilbert-space nature of \cref{alg:pd:alg}, in this section we talk about gradients $\thisgrad \in X_k$ and their estimates, while in most of the rest of the manuscript we talk about differentials $\thisdiff \in X_k^*$ and their estimates in general normed spaces. The former is the Riesz representation of the latter.
Often, even in Hilbert spaces, it is more convenient to not work with the presentation.
For the online solution of \cref{eq:pd:infinite-horizon-problem}, in this section, we recall the primal-dual method of \cite{tuomov2024online-eit}, presented in \cref{alg:pd:alg}, and its assumptions and regret theory.
\term{Regret} is replacement in online optimisation of the concept of \emph{convergence} in static optimisation. The rough idea is to bound the regret of past decisions (steps taken by the optimisation algorithm) subject to new information. In practise, mathematically, regret results are analogous to convergence results, but often weaker.

\Cref{alg:pd:alg} consists of three main steps:
\begin{enumerate}\item \emph{Predict} the primal and dual variables $\predicted{u}{k+1} = (\predicted{x}{k+1},\predicted{y}{k+1})$ for frame $k+1$, from the iterate $\thisu= (\thisx, \thisy)$ of frame $k$ (and earlier frames). This should use appropriate dynamical models; in the EIT experiments of \cite{tuomov2024online-eit} and \cref{sec:numerical}, the transport equation (optical flow) is used.
    \item \emph{Form} the gradient estimate  $\estgrad{E}{k+1}(\predicted{x}{k+1})$. Typically $E_{k+1}$ is the data term, and in our EIT application of interest, computing $\grad E_{k+1}(\predicted{x}{k+1})$ exactly would involve solving a PDE and an adjoint PDE. We want to avoid this. In the present paper, to do this, we will use single or few \emph{inner} and \emph{adjoint} steps that arise from standard algorithms. This is the topic of \cref{sec:inner-adjoint}.
    \item \emph{Correct} the primal and dual variables to $\nextu= (\nextx, \nexty)$ using a standard primal-dual proximal splitting step based on the predicted iterate, and employing the differential estimate.
\end{enumerate}
This algorithm was, thus, already designed with differential estimates in mind.
As we will see in this section, these estimates only need to satisfy simple smoothness-type properties with controlled error.
We will in the next \cref{sec:online-tracking} represent a rigorous theory of single-loop online differential estimation for forming estimates that satisfy those properties.\begin{algorithm}
    \caption{Nonconvex predictive online primal-dual proximal splitting (POPD-N)}
    \label{alg:pd:alg}
    \begin{algorithmic}[1]
        \Require For all $k\in\N$, on real Hilbert spaces $X_k$ and $Y_k$: convex, proper, lower semicontinuous  $F_{k+1}\colon X_{k+1}\to\extR$, and $G_{k+1}^*\colon Y_{k+1}\to\extR$; a primal-dual predictor $P_k\colon X_k\times Y_k\to X_{k+1}\times Y_{k+1}$; and an operator $K_{k+1}\in\linear(X_{k+1};Y_{k+1})$.
        Gradient estimates $\estgrad{E}{k+1}(\predicted{x}{k+1})$ for a $E_{k+1}\colon X_{k+1}\to\R$.
        Step length parameters $\tau,\sigma >0$.
        \State Pick initial iterates $x^0\in X_0$ and $y^0\in Y_0$.
        \For{$k\in\N$}
        \State
        \label{item:alg:pd:predict}
        $(\predicted{x}{k+1},\predicted{y}{k+1}) \defeq P_k(\thisx,\thisy)$
        \Comment{prediction step}
        \State
        \label{item:alg:pd:primal}
        $\nextx \defeq \prox_{\tau F_{k+1}}\!\bigl(\predicted{x}{k+1} - \tau \estgrad{E}{k+1}(\predicted{x}{k+1}) - \tau K_{k+1}^*\predicted{y}{k+1}\bigr)$
        \Comment{primal step}
        \State
        \label{item:alg:pd:dual}
        $\nexty \defeq \prox_{\sigma G_{k+1}^*}\!\bigl(\predicted{y}{k+1} + \sigma K_{k+1}(2\nextx - \predicted{x}{k+1})\bigr)$
        \Comment{dual step}
        \EndFor
    \end{algorithmic}
\end{algorithm}

\subsection{Assumptions and basic setup}

We need to make several assumptions.
Compared to \cite{tuomov2024online-eit}, for simplicity, we choose step length parameters and growth and smoothness factors to be independent of the frame number $k$.
To allow evolving data to be located anywhere in the problem, all the functions and operators in \eqref{eq:pd:infinite-horizon-problem}, and spaces, however, remain dependent on the frame index $k$.
We refer to \cite{tuomov2024online-eit} for the full theory.

The first assumption formally introduces the relevant functions, predictors, and comparison sets.\begin{assumption}[{Basic structural assumptions \cite[Assumption 2.1]{tuomov2024online-eit}}]
    \label{ass:pd:main}
    Let $N\in\N$. On real Hilbert spaces $X_k$ and $Y_k$ for $0\le k\le N$, we are given:
    \begin{enumerate}[label=(\roman*),nosep]
        \item Convex, proper, lower semicontinuous functions
              $F_k\colon X_k\to\extR$ and $G_k^*\colon Y_k\to\extR$,
              with uniform strong convexity factors $\gamma\ge0$ and $\rho\ge0$, respectively; operators $K_k\in\mathbb{L}(X_k;Y_k)$; and a possibly nonconvex but finite-valued function $E_k\colon X_k\to\R$.
        \item A bounded set $\PDpredictConstr \subset \ProductSpace$ of primal-dual comparison sequences. Here
              \[
                  \ProductSpace
                  \defeq
                  (X_0\times Y_0)\times(X_1\times Y_1)\times(X_2\times Y_2)\cdots.
              \]
        \item Primal-dual predictors $P_k\colon X_k\times Y_k\to X_{k+1}\times Y_{k+1}$ that produce $\predicted{u}{k+1} = P_k(u^k)$.
    \end{enumerate}
\end{assumption}

The predictor $P_k$ is merely a notational convenience. Unlike the notation suggests, we allow $(\predicted{x}{k+1}, \predicted{y}{k+1}) = P_k(x^k, y^k)$ to depend on the entire history, not merely $(x^k, y^k)$.

From the point of view of the satisfaction of our remaining assumptions, we would ideally be able to take $\PDpredictConstr$ as the whole space. However, our main regret result does not allow this.
Another practical option would be the set of solutions to \eqref{eq:pd:infinite-horizon-problem}, without observing the dynamic constraint $\primaldynamics$.
However, intuitively, both
\begin{equation}
    \label{eq:pd:PpredictConstr}
    \PpredictConstr \defeq \{\bar x^{0:\infty}\in X_{0:\infty}\ | (\bar x^{0:\infty},\bar y^{0:\infty})\in\PDpredictConstr\}
\end{equation}
and the predictors $P_k$ should approximate the original primal dynamics $\primaldynamics$.

Our second assumption formalises the requirement of a “smooth estimate” $\thisestgrad$ of each $\grad E_k(\this{\primalpredict})$.
The monotonicity-like conditions in \cref{item:pd:main-forwardstep-local0,item:pd:main-forwardstep-local} are superficially similar, but have slight differences: the first one is supposed to hold globally for any $x$, the second one, only locally with better factors. The regret bound proofs in \cite{tuomov2024online-eit} use the global version to derive an \emph{a priori} estimate that ensures that the iterates stay in the neighbourhood where the second \emph{a posteriori} estimate holds.
Verifying this assumption for single-loop estimates generated by standard inner and adjoint algorithms, is the main content of this paper in the following \cref{sec:inner-adjoint}.
In \cite{tuomov2024online-eit} a background solver was used to generate these estimates.
Moreover, \cite{tuomov2024online-eit} also formalises corresponding fully global assumption. For simplicity, we concentrate here on the local version, only.

\begin{assumption}[{Local smoothness and growth, and step length bounds \cite[Assumption 2.3]{tuomov2024online-eit}}]\footnotemark
    \label{ass:pd:smoothness}
    Let $(\predicted{x}{k},\predicted{y}{k}) = P_{k-1}(u^{k-1})$. Given $N \in \N$, \cref{ass:pd:main} holds, and
    \begin{enumerate}[label=(\roman*),nosep]
        \item\label{item:pd:main-forwardstep-local0}
              \textbf{Global a priori smoothness and growth:}
              For all $1 \le k \le N$,
              $E_k$ satisfies for some growth and smoothness factors $\EkGrowthGlobal \in \R$, $\EkLossGlobal \ge 0$, and errors $\ErrGlobal_k\ge0$, for all  $\this\realoptu = (\this{\realoptx}, \this{\realopty}) \in H_k^{-1}(0)$, and for all $x^k \in X_k$ the “erroneous” three-point monotonicity-like property
              \[
                  \iprod{\thisestgrad- \grad E_k(\this\realoptx)}{x^k - \this\realoptx} \ge {\EkGrowthGlobal}{}\norm{x^k - \this\realoptx}^2 - {\EkLossGlobal}{}\norm{x^k - \this\primalpredict}^2 - \ErrGlobal_k.
              \]
        \item\label{item:pd:main-forwardstep-local}
              \textbf{Local a posteriori smoothness and growth:}
              For every comparison sequence $\bar x^{1:N} \in \PpredictConstr_{1:N} \cap \prod_{i=1}^N B(\predicted{x}{i},\delta)$,
              for all $1 \le k \le N$,
              $E_k$ satisfies for some smoothness and growth factors $\EkLoss \ge 0$ and $\EkGrowthMono \ge \EkGrowth \ge 0$, errors $\Err_k, \ErrMono_k \ge 0$, and a radius $\delta>0$, for all $x^k \in B(\this{\bar x}, \delta)$, the inequality
              \begin{gather*}
                  \iprod{\thisestgrad}{x^k-\this{\bar x}}
                  \ge
                  E_k(x^k)-E_k(\this{\bar x})
                  + \frac{\EkGrowth}{2}\norm{x^k - \this{\bar x}}^2
                  - \frac{\EkLoss}{2}\norm{x^k - \this\primalpredict}^2 - \Err_k,
                  \intertext{as well as, for any  $\this\realoptu = (\this{\realoptx}, \this{\realopty}) \in H_k^{-1}(0) \isect B(\this\primalpredict, \delta) \times Y_k$ and $x^k \in B(\realoptx^k,{\delta})$ that}
                  \iprod{\thisestgrad - \grad E_k(\this\realoptx)}{x^k - \this\realoptx} \ge {\EkGrowthMono}{}\norm{x^k - \this\realoptx}^2 - {\EkLossMono}{}\norm{x^k - \this\primalpredict}^2 - \ErrMono_k.
              \end{gather*}
              Here the predicted primal variable $\predicted{x}{k}$ is given by \cref{ass:pd:main}.
        \item\label{item:pd:smoothness:steplength}
              \textbf{Step length factors:}
              The step length parameters $\tau,\sigma>0$ along with a $\EkLipCoeffCoeff\in(0,1)$, satisfy
              \begin{align}
                  \label{eq:pd:primaltestcond-positivity-local}
                  1 & \ge
                  \max \Bigl\{
                  \EkLoss,\,
                  -2(\gamma + \EkGrowthGlobal),\,
                  \tfrac{2}{1-\EkLipCoeffCoeff}\EkLossGlobal,\,
                  \tfrac{2}{1-\EkLipCoeffCoeff}\EkLossMono
                  \Bigr\}\tau
                  + \tau\sigma\norm{K_k}^2.
              \end{align}
    \end{enumerate}
\end{assumption}

\footnotetext{%
    Compared to \cite[Assumption 2.3]{tuomov2024online-eit}, this fixes the \term{testing parameters} $\phi_k \equiv 1$, $\eta_k \equiv \tau$, and $\psi_k \equiv \tau/\sigma$.

    Moreover, we use a slightly weaker version of the smoothness and growth assumptions. In
    \cite{tuomov2024online-eit}, the corresponding estimates are formulated for arbitrary test points in the indicated neighbourhoods and arbitrary framewise comparison points. In the present paper, we only require them at the specific primal iterate $x^k$ generated by the algorithm, and only for the $k$th component $\bar x^k$ of the comparison sequence under consideration.

    This restriction is sufficient for the arguments used for \cite[Theorem 2.12]{tuomov2024online-eit}. Indeed, in the theorem, the relevant estimates are only applied at the current iterate. Likewise, the comparison point is not used as an arbitrary isolated element of the framewise slice $\PpredictConstr_k$, but as the $k$th component of a fixed comparison sequence $\bar x^{1:N} \in \PpredictConstr_{1:N}$. More precisely, in the proof of \cite[Theorem 2.12]{tuomov2024online-eit}, $\bar x^k$ is the primal component of the $k$th component of the fixed comparison sequence $\bar u^{1:N} \in \mathcal U_{1:N}$. This is the reason for formulating the present assumption for comparison sequences $\bar x^{1:N} \in \PpredictConstr_{1:N}$ rather than for arbitrary framewise comparison points. Once such a sequence is fixed, the proof only uses its $k$th $x$ component at time $k$. Further, \cite[Assumption 2.5(i)]{tuomov2024online-eit} is applied in \cite[Lemma 2.10, before equation (17)]{tuomov2024online-eit} with $x=x^k$. The second estimate in \cite[Assumption 2.5(ii)]{tuomov2024online-eit} is applied in \cite[Lemma 2.11, before equation (24)]{tuomov2024online-eit}, again with $x=x^k$. The first estimate in \cite[Assumption 2.5(ii)]{tuomov2024online-eit} is applied in \cite[Theorem 2.12, equation (29)]{tuomov2024online-eit}, also with $x=x^k$.

    The remaining use of these constants is indirect. In \cite[Lemma 2.9]{tuomov2024online-eit}, they are used to prove positive definiteness of the testing operators, in particular of $\hat\Gamma_k$. For this step, it is enough to retain the coefficient ordering $\hat\gamma_{E,k}\ge \gamma_{E,k}$; the estimates themselves are not applied to arbitrary additional test points. This positive definiteness estimate is later used in \cite[Lemma 2.11]{tuomov2024online-eit} for
    the term $u^k-\bar u^k$.

    Hence, for the present verification, it is enough to prove the smoothness and growth inequalities at the realised algorithmic point $x^k$. The neighbourhood assumptions are still needed to ensure that the predictor, comparison point, and reference point lie in the region where the local estimates are valid.
}

Our final assumption relates the comparison set to the set of critical points of $H$, and sets bounds on step lengths.
To state the assumption, we need to introduce additional notation.
First of all, the first-order optimality conditions for the static problem $\min F_k+E_k+G_k \circ K_k$ for a single frame $k$ in \eqref{eq:pd:infinite-horizon-problem}, can be expressed with the general notation $u=(x,y)$ as \cite{clasonvalkonen2020nonsmooth}
\begin{equation}
    \label{eq:pd:h0}
    0 \in H_k(\this{\realoptu})
    \quad\text{for}\quad
    H_k(u) \defeq
    \begin{pmatrix}
        \subdiff F_k(x) + \grad E_k(x) + K_k^* y
        \\
        \subdiff G_k^*(y) - K_k x
    \end{pmatrix}
    .
\end{equation}
We also write
\[
    H_{n:m}(u^{n:m}) \defeq  H_n(u^n)\times\dots\times H_m(u^m)\subset \ProductSpace_{n:m}.
\]
\Cref{alg:pd:alg} may then be written in implicit form as the iterative solution of $\thisu$ from
\begin{gather*}
    0 \in \widetilde H_k(\thisu) + M_k(\thisu-\this\pdpredict),
    \quad
    \this\pdpredict \defeq P_k(\thisu),
    \shortintertext{where}
    \widetilde H_k(u) \defeq
    \begin{pmatrix}
        \subdiff F_k(x) + \thisestgrad + K_k^* y
        \\
        \subdiff G_k^*(y) - K_k x
    \end{pmatrix}
    \quad\text{and}\quad
    \Precond_k =
    \begin{pmatrix}
        \inv\tau \Id & - K_k^*
        \\
        -K_k         & \inv\sigma \Id
    \end{pmatrix}
    .
\end{gather*}
This formulation facilitates convergence and regret analysis.
We also define the growth operator
\[
    \GrowthOp \defeq
    \diag\bigl(
    (\gamma+\EkGrowth)  \Id,\;
    \rho \Id
    \bigr).
\]
Recalling that $\nexxt\pdpredict \defeq P_k(\thisu)$, we then define for\footnote{We use this notation when $A$ might not be positive semi-definite of self-adjoint, and therefore not define a (semi-)norm} $\fakenorm{x}_A^2 \defeq \iprod{Ax}{x}$ the \term{prediction errors}
\begin{equation}
    \label{eq:pd:prediction-error}
    \epsilon_{k+1}^{\dagger}(u^k, \bar u^{k:{k+1}})
    \defeq
    \frac{1}{2}\fakenorm{\nexxt\pdpredict-\overnextu}_{\tau\Precond_{k+1}}^2 - \frac{1}{2}\fakenorm{\this u-\this{\bar u}}_{\tau(\Precond_k+ \GrowthOp)}^2
    \quad\text{for all}\quad \bar u^{k:{k+1}} \in \PDpredictConstr_{k:k+1}.
\end{equation}
They measure the difference of deviation from a chosen comparison sequence between the current iterate and its prediction.
When $u^k$ and $\bar u^{k:{k+1}}$ are clear from the context, we abbreviate $\epsilon_{k+1}^{\dagger} \defeq \epsilon_{k+1}^{\dagger}(u^k, \bar u^{k:{k+1}})$.

\begin{assumption}[{Critical point proximity \cite[Assumption 2.6]{tuomov2024online-eit}}]
    \label{ass:pd:add-local}
    Let $N\in\N$ and suppose \cref{ass:pd:smoothness} holds.
    Then:
    \begin{enumerate}[label=(\roman*),nosep]
        \item\label{ass:pd:add-local-i}
              For constants $\criticalprox_k>0$ ($1\le k\le N$), the primal-dual comparison set $ \PDpredictConstr_{1:N}$ stays close to the critical points of $H_{1:N}$ in the sense that
              \[
                  \PDpredictConstr_{1:N}
                  \subset
                  \Bigl\{
                  \bar u^{1:N}\in\ProductSpace_{1:N}
                  \ \Bigm|\
                  \inf_{\hat u^{1:N}\in H^{-1}_{1:N}(0)}
                  \norm{\bar u^k - \hat u^k}_{M_k+\GrowthOpMono_k}
                  \le \criticalprox_k
                  \text{ for }1\le k\le N
                  \Bigr\}.
              \]
        \item\label{ass:pd:add-local-init-iterate}
              Let
              $
                  \deltaCoeff_k
                  \defeq
                  1
                  +2\tau\min\{\gamma + \EkGrowthGlobal,\EkLossGlobal\}
                  -\tau\sigma\norm{K_k}^2
                  >0,
              $
              where positivity follows from \eqref{eq:pd:primaltestcond-positivity-local}.
              Then for some $\YoungCoeff_k,\Delta>0$, $\tilde\delta\in(0,\delta)$, and for all $\bar u^{0:N}\in\PDpredictConstr_{0:N}$, we have\footnote{%
                  The infimum here is over $0 \le n \le N-1$, while in \cite{tuomov2024online-eit} it is over $0 \le n \le N$.
                  The former is sufficient, as \cite[(23) or (2.19) in arXiv version, in the proof of Lemma 2.11]{tuomov2024online-eit} only has to hold for $1 \le n \le N$, and avoids an inconvenient indexing of $\bar u^{0:N}$ in this definition when $n=N$.
                  In either case, passing to the limit $N \upto \infty$, the same final bounds will be assumed.
              }
              \[
                  \begin{split}
                      0 < d_N(\bar u^{0:N})
                      \defeq
                      \inf_{0\le n\le N-1}\biggl(
                       &
                      \frac{\theta_{n+1}(\delta-\tilde\delta)^2}{\xi_{n+1}}
                      - \frac{1+\Perturbr}{\YoungCoeff_{n+1}}\tau\criticalprox_{n+1}^2
                      - 2\epsilon_{n+1}^{\dagger}(u^n, \bar u^{n:{n+1}})
                      \\
                       &
                      - \sum_{k=1}^n\!\left(
                      \frac{1+\Perturbr}{2\EkLipCoeffCoeff}\tau\criticalprox_k^2
                      + \epsilon_k^{\dagger}(u^{k-1}, \bar u^{k-1:k})
                      + \tau\ErrMono_k
                      \right)
                      \biggr)
                  \end{split}
              \]
              and
              \[
                  (2 - \EkLipCoeffCoeff)(\inv\xi_k + 1)(\delta-\tilde\delta)^2
                  + 2\inv\theta_k\tau\ErrGlobal_k
                  \le \tilde\delta^2.
              \]
              Here $\kappa$ arises from \cref{ass:pd:smoothness}\,\cref{item:pd:smoothness:steplength}.
    \end{enumerate}
\end{assumption}

The first part of the assumption ensures that the comparison set, affected by problem dynamics, does not stray too far from the sequences of static solutions.
It holds for some  $\criticalprox_{n+1}>0$ subject to boundedness assumptions on $\inv H_{1:N}(0)$.
If we can \emph{choose} $\PDpredictConstr_{1:N}$ to include $\inv H_{1:N}(0)$, then also $\criticalprox_{n+1}=0$.
This, however, affects the prediction errors \eqref{eq:pd:prediction-error}.
If the prediction errors, as well as the smoothness errors $\tau\ErrMono_k$ and $\ErrGlobal_k$ from \cref{ass:pd:smoothness} are small (compared to $\tau$) and with a bounded sum, the second part can also be made to hold for some choices of $\tau$, $\sigma$, $\tilde\delta$, $\xi_k$, and $\Delta$.
(Note that definition of the prediction error $\epsilon_{k+1}^{\dagger}(u^k, \bar u^{k:{k+1}})$ also includes multiplication by $\tau$.)

\subsection{Regret}

Our main regret results talks about “temporal sub-infimal convolutions” of the original objective, and the dual dynamics of the problem.
To state the result, besides the primal projection \eqref{eq:pd:PpredictConstr} of the set of comparison sequences, also define the dual projection
\[
    \DpredictConstr
    \defeq
    \left\{\bar y^{0:\infty}\in Y_{0:\infty}\ \middle|\ (\bar x^{0:\infty},\bar y^{0:\infty})\in\PDpredictConstr\right\}.
\]
Then define
\begin{gather*}
    \mathring G_{1:N}(z^{1:N})
    \defeq
    \sup_{\tilde y^{1:N} \in \DpredictConstr_{1:N}}
    \sum_{k=1}^N
    \bigl[
        \iprod{z^{k}}{\tilde y^k} - \tau G_k^*(\tilde y^{k})
        \bigr],
    \quad\text{and}\quad
    Q_{1:N}(x^{1:N}) \defeq \sum_{k=1}^{N} \tau[F_k + E_k](\thisx).
\end{gather*}
Also let\footnote{
    In \cite{tuomov2024online-eit}, incorrectly, $G_{1:N}(z^{1:N}) \defeq \sum_{k=1}^{N} \tau G_k(\this z)$, however, same as here, $G_{1:N}$ is only ever used in discussions.
    With the definitions here, $G_{1:N}(Kx^{1:N})= \sum_{k=1}^{N} [\tau G_k^*]^*(\thisx)=\sum_{k=1}^{N} \tau G_k(K_k \thisx)$, consistently with $Q_{1:N}$.
}
\[
    K_{1:N}x^{1:N} \defeq (\tau_1K_1x^1,\ldots,\tau_NK_Nx^N)
    \quad\text{and}\quad
    G_{1:N}(z^{1:N}) \defeq \sum_{k=1}^{N} [\tau G_k^*]^*(\this z)=\sum_{k=1}^{N} \tau G_k(\this z/\tau).
\]
We have
$
    \mathring G_{1:N}
    \le
    G_{1:N} \infconv \delta_{\DpredictConstr_{1:N}}^*.
$
Thus $\mathring G_{1:N}$ behaves as a ``sub-infimal convolution'' of $G_{1:N}$ with the dual temporal dynamics.
In our companion paper \cite{online-regtheory}, we show that, actually,
\[
    \mathring G_{1:N}
    =
    G_{1:N} \infconv \delta_{\closure \conv \DpredictConstr_{1:N}}^*,
\]
where $\infconv$ denotes the infimal convolution, and $\delta_{\closure \conv \DpredictConstr_{1:N}}$ the indicator function of the closed convex hull of $\DpredictConstr_{1:N}$.
Its conjugate is the so-called support function of this set.

\begin{example}Take $G_k = \alpha\norm{\freevar}_{2,1}$, and $K_k=\grad$ in a suitable finite element space. Thus $G_k \circ K_k$ models a frame-wise total variation regulariser for the regularisation parameter $\alpha$.
    Then $\mathring G_{1:N}(K_{1:N} \freevar)$ becomes a spatiotemporal regulariser that bridges between the frame-wise total variation regularisers, and the dual dynamics encoded in $\closure \conv \DpredictConstr_{1:N}$.
\end{example}

We may now state the main regret result for \cref{alg:pd:alg}.
Until each index $N$, it bounds the values of a modified primal objective function that, compared to an initial segment of \eqref{eq:online:problem}, replaces $G_{1:N} + \delta_{\primaldynamics_{1:N}}$ by $\mathring G_{1:N}$.

\begin{corollary}[{\cite[Corollary 2.13]{tuomov2024online-eit}}]
    \label{cor:pd:main}
    Let $N\ge1$, and suppose \cref{ass:pd:main,ass:pd:add-local,ass:pd:smoothness} holds for $u^N=(x^N,y^N)$ generated by \cref{alg:pd:alg}, with initialisation $u^0\in X_0\times Y_0$.
    If
    \begin{equation}
        \label{eq:pd:main-gap:local:init}
        \frac12\norm{u^0 - \optu^0}_{\tau(\Precond_0 + \GrowthOp_0)}^2
        \le d_N(\optu^{0:N}),
    \end{equation}
    then
    \[
        \begin{split}
            \bigl[Q_{1:N}(x^{1:N}) + \mathring G_{1:N}(K_{1:N}x^{1:N})\bigr]
             &
            \le
            \sup_{\optx^{1:N}\in\PpredictConstr_{1:N}}
            \bigl[Q_{1:N}(\optx^{1:N}) + \mathring G_{1:N}(K_{1:N}\optx^{1:N})\bigr]
            \\
            \MoveEqLeft[-1]
            +
            \sup_{\optu^{0:N}\in\PDpredictConstr_{0:N}}
            \Bigl(
            \tfrac12\norm{u^0-\optu^0}_{\tau\Precond_0+\Gamma_0}^2
            + c_N(\optx^{1:N}, y^{1:N})
            + e_N^\Sigma(u^{0:N-1}, \optu^{0:N})
            \Bigr),
        \end{split}
    \]
    where the cumulative prediction and gradient error
    \[
        e_N^\Sigma(u^{0:N-1},\optu^{0:N})
        \defeq
        \sum_{k=0}^{N-1}\!\left(
        \epsilon_{k+1}^{\dagger}(u^k,\optu^{k:k+1})
        + \tau\Err_{k+1}
        \right).
    \]
    and the \emph{comparison set solution discrepancy}
    \[
        c_N(\optx^{1:N}, y^{1:N})
        \defeq
        \inf_{\tilde y^{1:N}\in\DpredictConstr_{1:N}}
        \iprod{K_{1:N}\optx^{1:N}}{y^{1:N}-\tilde y^{1:N}}
        + G_{1:N}^*(\tilde y^{1:N}) - G_{1:N}^*(y^{1:N}).
    \]
\end{corollary}\begin{remark}[{Comparison set solution discrepancy \cite[Remark 2.7 and Section 3]{tuomov-better-predict}}]
    We have $c_N \le 0$ when $G_k \circ K_k=\alpha \norm{\grad\freevar}_{2,1}$ is the total variation regulariser, the dual initialisation achieves the total variation in the sense that $\iprod{y^0}{x^0}=\alpha \norm{\grad x^0}_{2,1}$ with $\norm{y^0}_{2,\infty} \le \alpha$, and the dual predictor is total variation preserving.
    Examples of such predictors are found in \cite{tuomov-better-predict}.
\end{remark}

\section{Online differential estimation}
\label{sec:online-tracking}

Our idea for a low-cost estimate of $\thisdiff=J'(S_{k,u}(\this{\primalpredict}))S_{k,u}'(\this{\primalpredict})$ is grounded in the work of \cite{jensen2022nonsmooth,tuomov2024tracking} for static problems:
on each iteration of the outer \cref{alg:pd:alg} for the solution of \eqref{eq:online:problem}---for each new data frame as it arrives---we take only a \emph{single step} or few steps of standard solvers, such as Gauss–Seidel splitting, towards the solution of \eqref{eq:eit:cem} and its corresponding adjoint equation.
As we do this on each iteration, \emph{if} the data does not change too abruptly between frames, we expect to still get convergence, or, in case of online optimisation, bounded regret.

We start by deriving the general scheme of approximation from adjoint equations in \cref{sec:online-tracking:adjoint}.
Then in \cref{sec:online-tracking:assumptions} we formulate an \term{online tracking assumption} that specific algorithms for the construction of the estimates have to satisfy.
In \cref{sec:online-tracking:smoothness} we use the tracking assumption to prove smoothness properties of the estimates $\thisestdiff$ of $\thisdiff$, in particular, to prove \cref{ass:pd:smoothness}. Finally, we study quadratic $E_k$ in detail in \cref{sec:online-tracking:smoothness-quadratic}.
In the next \cref{sec:inner-adjoint} we will prove the tracking assumption for specific inner and adjoint algorithms.
To ensure the applicability of our results to a broad class of problems---including PDE-constrained and dynamic inverse problems---we work in abstract normed spaces.
Our theory is also applicable to online bilevel optimisation.

\subsection{Adjoint equations and the general approach}
\label{sec:online-tracking:adjoint}

Let \( J_k: U_k \to \R \) and \( S_{k,u}: X_k \to U_k \), for \(k \in \{1, \ldots, n\}\), be Fréchet differentiable mappings on normed spaces \(X_k\) and \(U_k\).
We estimate
\begin{equation}
    \label{eq:online:derivative}
    \thisdiff
    = J_k'(S_{k,u}(\predicted{x}{k}))\, S_{k,u}'(\predicted{x}{k})
\end{equation}
by approximating \( S_{k,u}(\predicted{x}{k}) \) with a computable state \( u^k \in U_k \) and its derivative \( S_{k,u}'(\predicted{x}{k}) \) by a \( p^k \in \linear(X_k; U_k) \).
However, as $ p^k \approx S_{k,u}'(\predicted{x}{k}) \in \linear(X_k; U_k)$ can be very high-dimensional, we want to avoid constructing it directly.
Instead, we seek to only construct the necessary projections $w^k$.
This is illustrated by the next lemma, where $T_k$ can model, e.g., a PDE or the first-order optimality conditions of the inner optimisation problem \eqref{eq:tracking:implicit-inner} (in which case $T_k = E_k'$), both parametrised by $x$.
The proof follows standard techniques for deriving reduced adjoint equations of PDEs; see, e.g., \cite[§1.6.2]{hinze2009pde}.\begin{lemma}[Reduced adjoint]
    \label{ex:tracking:pde-adjoint}
    Let $k \in \N$ and $T_k: U_k \times X_k \to W_k^*$ be Fréchet differentiable, and suppose that $S_{k,u}(x)$ arises from the satisfaction of \eqref{eq:tracking:implicit-inner}.
    If $S_{k,u}$ is Fréchet differentiable in a neighbourhood of $x$, then
    \[
        E_k'(x) = w \diffwrt{T_k}{x}(S_{k,u}(x), x),
    \]
    where $ w \in W_k \hookrightarrow W_k^{**}$ (assuming sufficient regularity that a solution in $W_k$ exists) is a solution of the \term{reduced adjoint}
    \begin{equation}
        \label{eq:tracking:reduced-adjoint}
        w \diffwrt{T_k}{u}(S_{k,u}(x), x) + J_k'( S_{k,u}(x)) = 0.
    \end{equation}
\end{lemma}\begin{proof}By implicit differentiation and \eqref{eq:tracking:implicit-inner} holding for $x$ in neighbourhood of $x$, we get the basic adjoint
    \begin{equation}
        \label{eq:tracking:basic-adjoint}
        0 = \diffwrt{T_k}{u} ( S_{k,u}(x), x )S_{k,u}'(x) + \diffwrt{T_k}{x}(S_{k,u}(x), x),
    \end{equation}
    where
    $S_{k,u}'(x) \in \linear(X_k; U_k)$,
    $\diffwrt{T_k}{u}(S_{k,u}(x), x) \in \linear(U_k; W_k^*)$, and
    $\diffwrt{T_k}{x}(S_{k,u}(x), x) \in \linear(X_k; W_k^*)$.
    Applying $w$ from the left to \eqref{eq:tracking:basic-adjoint}, and  $S_{k,u}'(x)$ from the right to \eqref{eq:tracking:reduced-adjoint}, we deduce that
    $E_k'(x) = J_k'( S_{k,u}(x))S_{k,u}'(x) = w \diffwrt{T_k}{x}(S_{k,u}(x), x)$.
\end{proof}

In practice, we will take $w^k \in W_k$ as an approximate operator splitting solution of \eqref{eq:tracking:reduced-adjoint}, instantiated at $u^k \approx S_{k,u}(\predicted{x}{k})$ in place of $S_{k,u}(\predicted{x}{k})$, and then set
\[
    \thisestdiff \defeq w^k \diffwrt{T_k}{x}(u^k, \predicted{x}{k}) \approx J_k'(S_{k,u}(\predicted{x}{k}))S_{k,u}'(\predicted{x}{k}) = E_k'(\predicted{x}{k}).
\]
When $X_k$ is a real Hilbert space (such as $\R^n$), we write $\thisestgrad$ for the Riesz representation of $\thisestdiff$.

\textbf{This is our recursively constructed differential estimate}, which we wish to use in \cref{alg:pd:alg}.
We just need to show that it satisfies \cref{ass:pd:smoothness}.
We do this in \cref{sec:online-tracking:smoothness}, given \emph{tracking estimates} for $u^k$ and $w^k$, derived from the contractivity of splitting methods as in \cite{jensen2022nonsmooth,suonpera2024general}.
We state those tracking estimates in \cref{sec:online-tracking:assumptions}, and provide examples of their satisfaction in \cref{sec:inner-adjoint}.

\begin{example}\label{ex:adjoint:eit}
    We apply \cref{ex:tracking:pde-adjoint} to the EIT problem discussed in \cref{sec:intro}. Let $\sigma\in L^\infty(\Omega)$ and $U^k \in \R^{\Nelec}$ be fixed, and for simplicity assume $\Nmeas = 1$. Define the mixed state space $Y\defeq H^1(\Omega)\times\R^{\Nelec}$, $y=(u,I)\in Y$. Introduce the constraint operator
    \[
        T_k:Y\times L^\infty(\Omega)\to Y^*\quad
        \iprod{ T_k(y,\sigma)}{w}
        \defeq B_\sigma(y,w)-L_{U^k}(w),
    \]
    where $B_\sigma:(H^1(\Omega)\times\mathbb{R}^{\Nelec})\times
        (H^1(\Omega)\times\mathbb{R}^{\Nelec}) \to \mathbb{R}$
    for $w = (v,J) \in H^1(\Omega) \times \R^{\Nelec}$ is given by
    \[
        \begin{aligned}
            B_\sigma((u,I),(v,J))
             & \defeq
            \int_\Omega \sigma\nabla u\cdot\nabla vdx
            +\sum_{i=1}^{\Nelec} \frac{1}{\zeta_i}\int_{e_i} uvds
            -\sum_{i=1}^{\Nelec} \frac{J_i}{\zeta_i}  \int_{e_i} uds
            +\sum_{i=1}^{\Nelec} J_i I_i
        \end{aligned}
    \]
    and $L_{U^k}: H^1(\Omega)\times\mathbb{R}^{\Nelec} \to \R$ by
    \begin{equation}
        L_{U^k}(v,J)
        \defeq
        \sum_{i=1}^{\Nelec} \frac{1}{\zeta_i}  \int_{e_i} U_i(v-J_i)ds.
    \end{equation}
    The weak solution of \eqref{eq:eit:cem}, i.e. the forward state $S_{k,u}(\sigma)=y(\sigma, U^k)$, satisfies $T_k(S_{k,u}(\sigma),\sigma)=0$. Moreover
    $
        \iprod{\partial_y T_k(y,\sigma)[h_y]}{w}
        = B_\sigma(h_y,w)
    $
    and
    $
        \iprod{\partial_\sigma T_kk(y,\sigma)[h_\sigma]}{w}
        =\int_\Omega h_\sigma \nabla u\cdot\nabla v dx
    $
    to all directions $h_y \in Y$ and $h_\sigma \in L^\infty(\Omega)$. Taking $E_k(\sigma)=J_k(S_{k,u}(\sigma))$ with $J_k(y)=\frac12 \norm{\WOp(I - \EITmeas^k)}_2^2$ gives the data misfit term of \eqref{eq:eit:functionals}. For this term, \cref{ex:tracking:pde-adjoint} yields an adjoint state $w(\sigma, U^k)=(v,V)\in Y$ solving
    \[
        B_\sigma(\tilde y,w)
        =
        -J_k'(S_{k,u}(\sigma))[\tilde y]
        =
        \tilde I^\top  \Sigma^{-1} (\EITmeas^k - I(\sigma,U^k)) ,
        \quad\text{for all}\quad\,\tilde y = (\tilde u, \tilde I)\in Y.
    \]
    The differential is then $E_k'(\sigma)=\nabla v\cdot\nabla u$.
\end{example}

\subsection{Basic constructions and assumptions}
\label{sec:online-tracking:assumptions}

Recall that our goal is, at each time and iteration index $k$, to construct $(u^k, w^k) \approx \big(S_{k,u}(\predicted{x}{k}), S_{k,w}(\predicted{x}{k})\big)$, since the latter may be computationally infeasible or expensive to compute. Indeed, $S_{k,u}$ is defined as the solution of an inner problem, such as a PDE or an optimization problem, and solving it exactly at every iteration can be expensive.
Therefore, in practice, we work with approximate solutions $(u^k, w^k)$.
This naturally introduces approximation errors, whose behaviour must be controlled in order to ensure stability and convergence of the overall method. The following principal assumption formalises this requirement in the online setting, and can be seen as a counterpart of the corresponding offline assumptions in \cite{suonpera2024general,tuomov2024tracking}.\begin{assumption}\label{ass:tracking:main}
    For all $k \in \N$, let the normed space $X_k$ be equipped with a semi-norm $\norm{\cdot}_{k, \circ}$, and $X_k^*$ with the corresponding support function (see \cite[Section 2]{tuomov2024tracking})
    \[
        \norm{x^*}_{k,*} \defeq \sup\{ \dualprod{x^*}{x}_{X^*,X} \mid  x \in X,\, \norm{x}_{k,\circ} \le 1 \}.
    \]
    Also let the abstract spaces $U_k$, and $W_k$ be equipped with
    the distance functions $d_{U_k}: U_k \times U_k \to [0,\infty]$, and $d_{W_k}: W_k \times W_k \to [0,\infty]$.
    Let $\localset_k \subset X_k$, and let
    $S_{k,u}: X_k \to U_k$ and
    $S_{k,w}: X_k \to W_k$
    denote the inner and adjoint solution maps, respectively.
    Then, on each iteration $k \in \N$, given $((x^n, \predicted{x}{n}, u^n, w^n))_{n=0}^k \in \prod_{n=0}^k \localset_n\times \localset_n \times U_n \times W_n$ and the prediction $\predicted{x}{k+1} \in \localset_{k+1}$:
    \begin{enumerate}[label=(\roman*)]
        \item\label{item:tracking:main:inner-tracking}
              An \term{inner algorithm} produces $\nextu \in U_{k+1}$ satisfying, for some $\pi_u>0$, $\kappa_u>1$, $\innerError{k} \geq 0$, \emph{when $k \ge 1$}, the inner tracking inequality
              \begin{align*}
                  \kappa_u \dis{U_{k+1}}{u^{k+1}}{S_{k+1,u}(\predicted{x}{k+1})}
                   &
                  \le
                  \dis{U_k}{u^{k}}{S_{k,u}(\predicted{x}{k})}
                  + \pi_u\norm{x^{k} - \predicted{x}{k}}_{k, \circ} + \innerError{k}.
              \end{align*}

        \item\label{item:tracking:main:adjoint-tracking}
              An \term{adjoint algorithm} produces $\nexxt w \in W_{k+1}$ satisfying  for some ${\primaldifffact}, \pi_w>0$, $\kappa_w>1$, $\adjointError{k} \geq 0$, \emph{when $k \ge 1$}, the adjoint tracking inequality
              \begin{align*}
                  \kappa_w \dis{W_{k+1}}{\nexxt w}{S_{k+1,w}(\predicted{x}{k+1})}
                   &
                  \le
                  \dis{W_k}{\this w}{S_{k,w}(\predicted{x}{k})}
                  + {\primaldifffact} \dis{U_{k+1}}{u^{k+1}}{S_{k+1,u}(\predicted{x}{k+1})}
                  \\
                  \MoveEqLeft[-1]
                  + \pi_w\norm{x^{k} -  \predicted{x}{k}}_{k, \circ}
                  + \adjointError{k}.
              \end{align*}

        \item\label{item:tracking:main:differential-transformation}
              A \term{differential transformation}, $(u^{k+1}, w^{k+1})\in U_{k+1}\times W_{k+1}$, given  produces $\nextestdiff \in X_{k+1}^*$ that satisfies for some $\alpha_u,\alpha_w \ge 0$, $\transformError{k} \geq 0$ the bound
              \begin{align*}
                  \norm{\nextestdiff - \nextdiff}_{k+1, *}
                   &
                  \le
                  \alpha_u \dis{U_{k+1}}{\nexxt u} {S_{k+1,u}(\predicted{x}{k+1})}
                  + \alpha_w \dis{W_{k+1}}{\nexxt w}{S_{k+1,w}(\predicted{x}{k+1})} + \transformError{k}.
              \end{align*}
    \end{enumerate}
\end{assumption}\begin{remark}
    We observe regarding \cref{ass:tracking:main}:
    \begin{enumerate}\item If $\norm{\freevar}_{k,\circ}$ is a norm, then $\norm{\freevar}_{k,*}$ is the dual norm.
        \item The inner and adjoint tracking conditions
              \cref{item:tracking:main:inner-tracking,item:tracking:main:adjoint-tracking}
              are parameter change aware contractivity conditions for the inner and adjoint algorithms:
              if $\predicted{x}{k}=\thisx$ and $S_{k+1,u}=S_{k,u}$ (i.e., there is no data updates between frames), the former reduces to a standard contractivity condition.
              The differential transformation condition
              \cref{item:tracking:main:differential-transformation}
              bounds the construction error of $\nextestdiff$ in terms of the tracking errors of the iterates relative to the corresponding solution maps.
        \item  The distance functions $d_{U_k}, d_{W_k}$ will in the simplest case be given by norms or semi-norms. More generally, they may be problem-adapted measures of discrepancy between iterates and solution maps. Their exact form plays no role in the theory of this section.

        \item The spaces $X_k$ and $U_k$ that depend on the time and iteration index $k$ allow our framework to accommodate dynamic online optimisation
              \cite{hall13dynamical,tuomov-better-predict,tuomov2024online-eit}
              with growing data sets.

        \item For the inner and adjoint variables, the tracking inequalities are formulated starting from $k \ge 1$, and thus do not explicitly involve the initial iterates. At each iteration, the approximate inner and adjoint variables are computed using the prediction $\predicted{x}{k}$ and compared against the corresponding exact solution maps $S_{k,u}(\predicted{x}{k})$ and $S_{k,w}(\predicted{x}{k})$.
              The role of the prediction is to provide a consistent reference point across iterations, allowing the tracking errors to be propagated even when the underlying spaces vary with $k$. The initial discrepancies at the first iteration enter the final estimates, but can be controlled by computing the first step with sufficiently high accuracy. For contractive algorithms, these initial errors are naturally reduced relative to the initialisation.
    \end{enumerate}
\end{remark}

\subsection{Smoothness of differential estimates}
\label{sec:online-tracking:smoothness}

In this section, we derive descent and Lipschitz type inequalities for the differential estimates $\estdiff{E}{k}$, given the tracking estimates.
These results are straightforward consequences of the scalar tracking estimates of \cref{sec:technical}.
We recall that if $E$ is convex with $E'$  $L$-Lipschitz, then it satisfies the three point descent inequality \cite[Corollary~7.2]{clasonvalkonen2020nonsmooth}
\[
    \dualprod{\nextdiff}{x - \bar x}_{X_k^*,X_k} \ge E_{k+1}(x) - E_{k+1}(\bar x) - \frac{L}{2} \norm{x - \predicted{x}{k+1}}^2_{k, \circ}, \quad\text{for all}\quad  x,  \predicted{x}{k+1}, \bar x\in X_{k+1},  k\geq 0.
\]
Summing this estimate with the claim of the next lemma, provides a similar estimate for $\nextestdiff$, with additional error terms.
In \cref{thm:tracking:smoothness-verification}, after the lemma, we derive such estimates without assuming the convexity of $E$. Those derivations improve upon the basic non-convex counterparts of \cite[Appendix B]{tuomov-proxtest} to the above inequality.

\begin{lemma}\label{lemma:tracking:smoothness-pre-result}
    Suppose that \cref{ass:tracking:main} holds, and that $((x^n, \predicted{x}{n}, u^n, w^n))_{n=0}^k \in \prod_{n=0}^k \localset_n\times \localset_n \times U_n \times W_n$ as well as $\predicted{x}{k+1} \in \localset_{k+1}$ for some $k \in \N$.
    Then, for any $\tilde\gamma>0$, and $x, \bar x\in X_{k+1}$, we have
    \begin{equation}
        \label{eq:tracking:smoothness-pre-result}
        \begin{split}
            \dualprod{\estdiff{E}{k+1}(\predicted{x}{k+1}) - E_{k+1}'(\predicted{x}{k+1})}{x - \bar x}_{X_{k+1}^*,X_{k+1}}
             &
            \ge
            - \frac{\tilde\gamma}{2} \norm{x - \bar x}^2_{k+1,\circ}
            -\frac{5}{4\tilde\gamma}\trackingressum[1]^2 \norm{\predicted{x}{k+1} - x^{k+1}}_{k+1,\circ}^2
            - \frac{1}{2\tilde\gamma} \combinederror_k,
        \end{split}
    \end{equation}
    where
    \[
        \combinederror_k\defeq \frac{5}{2}\trackingressum[1]^2 \left(\max\!\left( \frac{1}{\pi_u}\,\innerError{k}, \frac{1}{\pi_w}\,\adjointError{k} \right)\right)^2
        + e_{1,k}
    \]
    for some $e_{1,k} \in \R$ and $\trackingressum[1] \ge 0$ that,
    with $\kappa \defeq \min (\kappa_u, \kappa_w)$ and $\MAX\kappa \defeq \max (\kappa_u, \kappa_w)$, satisfy
    \begin{gather}
        \label{eq:tracking:smoothness-pre-result:e-bound}
        \begin{split}
            \sum_{n=0}^{k} e_{1,n}
             &
            \le
            \Psi_1
            \defeq
            \frac{5}{4}\left(
            \frac{d_{U_1}^1(u^1, S_{1,u}(\predicted{x}{1}))}{\pi_u} \bigg(\frac{\trackingressum[1]\alpha_u\kappa}{\kappa-1} + \frac{\trackingressum[1]\alpha_w{\primaldifffact}}{(\kappa-1)^2}\bigg)
            +
            \frac{d_{W_1}^2(w^1, S_{1,w}(\predicted{x}{1}))}{\pi_w} \bigg(\frac{\trackingressum[1]\alpha_w\kappa}{\kappa-1}\bigg)
            +
            \sum_{n=0}^{k} \transformError{n}^2
            \right)
        \end{split}
        \shortintertext{and}
        \label{eq:tracking:smoothness-pre-result:varrho-bound}
        \begin{split}
            \trackingressum[1]
             &
            \le
            \frac{(\alpha_u\pi_u+\alpha_w\pi_w)\kappa\MAX\kappa}{\kappa-1}
            + \frac{\alpha_w {\primaldifffact}\pi_u\MAX\kappa}{(\kappa-1)^2}.
        \end{split}
    \end{gather}
\end{lemma}\begin{proof}By \cite[Lemma 2.1 (iii)]{tuomov2024tracking}, we have $\dualprod{x^*}{x}_{X_k^*,X_k} \le \frac{1}{2}\norm{x^*}_{k,*}^2 + \frac{1}{2}\norm{x}_{k,\circ}^2$ for all $x \in X_k$ and $x^* \in X_k^*$.
    Hence,
    \begin{equation}
        \label{eq:coro00}
        \dualprod{\estdiff{E}{k+1}(\predicted{x}{k+1}) - E_{k+1}'(\predicted{x}{k+1})}{x - \bar x}_{X_{k+1}^*,X_{k+1}}
        \ge
        - \frac{1}{2\tilde\gamma}\norm{\estdiff{E}{k+1}(\predicted{x}{k+1}) - E_{k+1}'(\predicted{x}{k+1})}^2_{k+1, *} - \frac{\tilde\gamma}{2}\norm{x - \bar x}^2_{k+1, \circ}.
    \end{equation}
    Let
    \begin{align*}
        d_k^u          & \defeq d_{U_k}(u^k, S_{k,u}(\predicted{x}{k})),
                       &
        \varrho_k      & \defeq \norm{\predicted{x}{k} - x^k}_{k,\circ},
                       &
        \epsilon_k     & \defeq \max\left(
        \innerError{k}/\pi_u,
        \adjointError{k}/\pi_w
        \right),
        \\
        d_k^w          & \defeq d_{W_k}(w^k, S_{k,w}(\predicted{x}{k})),
                       &
        \rlap{\text{and}}
                       &
                       &
        \widetilde d_k & \defeq \norm{\estdiff{E}{k+1}(\predicted{x}{k+1}) - E_{k+1}'(\predicted{x}{k+1})}_{k+1, *}.
    \end{align*}
    Then \cref{ass:tracking:main} and our assumptions  $((x^n, \predicted{x}{n}, u^n, w^n))_{n=0}^k \in \prod_{n=0}^k \localset_n\times \localset_n \times U_n \times W_n$ as well as $\predicted{x}{k+1} \in \localset_{k+1}$ imply \cref{ass:online-scalar-tracking:main} for the same $ \transformError{k}$.
    Thus, \cref{lemma:weaker:error-sum} gives
    \[
        \begin{split}
            \norm{\estdiff{E}{k+1}(\predicted{x}{k+1}) -  E_{k+1}'(\predicted{x}{k+1})}_{k+1,*}^2
             &
            \le
            \frac{5}{4}\trackingressum[1]^2(\varrho_{k+1}+\epsilon_{k+1})^2 + e_{1,k}
            \\
             &
            \le
            \frac{5}{2}\trackingressum[1]^2 \varrho_{k+1}^2
            + \frac{5}{2}\trackingressum[1]^2\left(\max\!\left( \frac{1}{\pi_u}\,\innerError{k}, \frac{1}{\pi_w}\,\adjointError{k} \right)\right)^2
            + e_{1,k}
            \\
             & =
            \frac{5}{2}\trackingressum[1]^2 \norm{\predicted{x}{k+1} - x^{k+1}}_{k+1,\circ}^2
            + \combinederror_k.
        \end{split}
    \]
    Combining this inequality with \cref{eq:coro00} proves \eqref{eq:tracking:smoothness-pre-result}.
    The bounds \cref{eq:tracking:smoothness-pre-result:e-bound,eq:tracking:smoothness-pre-result:varrho-bound} follow from \cref{lemma:weaker:inner-product-error-estimate,lemma:weaker:error-sum}.
\end{proof}

The next result, which verifies \cref{ass:pd:smoothness}, relates the error $\combinederror_k$ to the error expression in the regret \cref{cor:pd:main} for the outer method. We will further develop the complete error expression in \cref{rem:inner-adjoint:error}.
The result depends on the exact differential each $E_k$ essentially satisfying \cref{ass:pd:smoothness}.
We will in the next \cref{sec:online-tracking:smoothness-quadratic} verify this for quadratic energies.

\begin{theorem}[{Verification of \cref{ass:pd:smoothness}}]
    \label{thm:tracking:smoothness-verification}
    Let $k \ge 1$. Suppose \cref{ass:tracking:main} holds and that $((x^n, \predicted{x}{n}, u^n, w^n))_{n=0}^{k-1} \in \prod_{n=0}^{k-1} \localset_n\times \localset_n \times U_n \times W_n$ as well as $\predicted{x}{k} \in \localset_{k}$ for some $k \in \N$.
    Further, suppose $E_k'$ is $L'_k$-Lipschitz with respect to $\norm{\cdot}_{k,\circ}$ in the sense that
    \[
        \norm{E_k'(u)-E_k'(v)}_{k,*}
        \le
        L'_k\norm{u-v}_{k,\circ}
        \quad\text{for all } u,v\in X_k,
    \]
    and that for some constants $\gamma_1, \gamma_2, L_1, L_2 > 0$ and that for some $\delta > 0$, every comparison sequence $\bar x^{1:k} \in \PpredictConstr_{1:k} \isect \prod_{n=1}^k B(\predicted{x}{n},\delta)$ and all $x \in B(\bar x^k, \delta)$, we have
    \begin{align}
        \label{eq:true-3pt-ii-hat-ii}
        \dualprod{E_k'(\predicted{x}{k})}{x - \bar x^k}_{X_k^*,X_k}
         & \ge
        E_k(x) - E_k(\bar x^k)
        +
        \gamma_1\norm{x - \bar x^k}_{k,\circ}^2
        -
        L_1
        \norm{x - \predicted{x}{k}}_{k,\circ}^2,
        \\
        \shortintertext{and for any $\this\realoptu = (\this{\realoptx}, \this{\realopty}) \in H_k^{-1}(0)$ and all $x \in B(\realoptx^k, \delta)$ we have}
        \label{eq:true-3pt-i-hat-ii}
        \dualprod{E_k'(\predicted{x}{k}) - E_k'(\realoptx^k)}{x - \realoptx^k}_{X_k^*,X_k}
         & \ge
        \gamma_2\norm{x - \realoptx^k}_{k,\circ}^2
        -
        L_2
        \norm{x - \predicted{x}{k}}_{k,\circ}^2.
    \end{align}
    Let $\combinederror_{k-1}$ and $\trackingressum[1]$ be as in \cref{lemma:tracking:smoothness-pre-result}.
    Then the following is true
    \begin{enumerate}[label=(\roman*)]
        \item \label{item:tracking:smoothness-verification:global}
              For any $\tilde\gamma>0$,
              \cref{ass:pd:smoothness} \ref{item:pd:main-forwardstep-local0} holds with
              \[
                  \bar \gamma_E = - L_k' - \frac{\tilde\gamma}{2}(L_k' + 1),
                  \quad
                  \bar \lambda_E = \frac{L_k'}{2\tilde\gamma} + \frac{5}{4\tilde\gamma}\trackingressum[1]^2,
                  \quad\text{and}\quad
                  \bar e_k = \frac{1}{2\tilde\gamma} \combinederror_{k-1},
              \]
              That is, for any $\this\realoptu = (\this{\realoptx}, \this{\realopty}) \in H_k^{-1}(0)$ and for all $x^k \in X_k$, we have
              \begin{equation}
                  \label{eq:err-iii-final}
                  \dualprod{\estdiff{E}{k}(\predicted{x}{k}) - E_{k}'(\realoptx^k)}{x^k - \realoptx^k}_{X_k^*,X_k}
                  \ge
                  \bar \gamma_E \norm{x^k - \realoptx^k}_{k,\circ}^2
                  -
                  \bar \lambda_E \norm{x^k - \predicted{x}{k}}_{k,\circ}^2
                  -
                  \bar e_k.
              \end{equation}

        \item \label{item:tracking:smoothness-verification:local}
              For $0 < \tilde\gamma < 2\min\{\gamma_1,\gamma_2\}$,
              \cref{ass:pd:smoothness} \ref{item:pd:main-forwardstep-local} holds with
              \[
                  \hat \gamma_E \defeq \gamma_2 - \frac{\tilde\gamma}{2},
                  \quad
                  \gamma_E \defeq \min \{\hat \gamma_E, 2\gamma_1 - \tilde\gamma \},
                  \quad
                  \lambda_E \defeq 2L_1 + \frac{5\trackingressum[1]^2}{2\tilde\gamma},\quad
                  \hat \lambda_E \defeq L_2 + \frac{5\trackingressum[1]^2}{4\tilde\gamma},
                  \quad
                  e_k \defeq \hat e_k \defeq \ \frac{1}{2\tilde\gamma} \combinederror_{k-1}.
              \]
              That is, for every $\bar x^{1:k} \in \PpredictConstr_{1:k} \isect \prod_{n=1}^N B(\predicted{x}{n},\delta)$, and for all $x^k \in B(\bar x^k,\delta)$, we have
              \begin{gather}
                  \label{eq:err-ii-final}
                  \dualprod{\estdiff{E}{k}(\predicted{x}{k})}{x^k - \bar x^k}_{X_k^*,X_k}
                  \ge
                  E_k(x^k) - E_k(\bar x^k)
                  +
                  \frac{\gamma_E}{2}
                  \norm{x^k - \bar x^k}_{k,\circ}^2
                  -
                  \frac{\lambda_E}{2}
                  \norm{x^k - \predicted{x}{k}}_{k,\circ}^2
                  -e_k.
                  \intertext{and, for any $\this\realoptu = (\this{\realoptx}, \this{\realopty}) \in H_k^{-1}(0) \isect B(\this\primalpredict, \delta) \times Y_k$ and $x^k \in B(\realoptx^k,\delta)$,}
                  \label{eq:err-i-final}
                  \dualprod{\estdiff{E}{k}(\predicted{x}{k}) - E'_k(\realoptx^k)}{x^k - \realoptx^k}_{X_k^*,X_k}
                  \ge
                  \hat \gamma_E \norm{x^k - \realoptx^k}_{k,\circ}^2
                  -
                  \hat \lambda_E \norm{x^k - \predicted{x}{k}}_{k,\circ}^2
                  -\hat e_k.
              \end{gather}
    \end{enumerate}
\end{theorem}

\begin{proof}\cref{item:tracking:smoothness-verification:global}:
    Since \cref{ass:tracking:main} holds and $((x^n, \predicted{x}{n}, u^n, w^n))_{n=0}^{k-1} \in \prod_{n=0}^{k-1} \localset_n\times \localset_n \times U_n \times W_n$ as well as $\predicted{x}{k} \in \localset_{k}$,  \cref{lemma:tracking:smoothness-pre-result} applied with index $k-1$ together with \cite[Lemma 2.1 (iii)]{tuomov2024tracking} yields that for any $\tilde\gamma>0$, any $x, z \in X_k$,
    \begin{equation}
        \begin{split}
            \label{eq:err-bounds-i}
            \dualprod{\estdiff{E}{k}(\predicted{x}{k}) - E_{k}'(\predicted{x}{k})}{x - z}_{X_k^*,X_k}
             &
            =
            -
            \adaptdualprod{
                -\tilde\gamma^{-1/2}
                \bigl(
                \estdiff{E}{k}(\predicted{x}{k}) - E'_k(\predicted{x}{k})
                \bigr)
            }{
                \tilde\gamma^{1/2}(x - z)
            }_{X_k^*,X_k}
            \\
             &
            \ge
            -
            \frac{1}{2\tilde\gamma}
            \norm{\estdiff{E}{k}(\predicted{x}{k}) - E'_k(\predicted{x}{k})}_{k,*}^2
            -
            \frac{\tilde\gamma}{2}
            \norm{x - z}_{k,\circ}^2
            \\
             &
            \ge
            - \frac{\tilde\gamma}{2} \norm{x - z}^2_{k, \circ}
            - \frac{5}{4\tilde\gamma}\trackingressum[1]^2 \norm{\predicted{x}{k} - x^k}_{k,\circ}^2
            - \frac{1}{2\tilde\gamma} \combinederror_{k-1}.
        \end{split}
    \end{equation}
    Since $E'_k$ is $L_k'$-Lipschitz between the semi-norm $\norm{\cdot}_{k,\circ}$ and support function $\norm{\cdot}_{k,*}$, and the two satisfy a Cauchy--Schwarz inequality \cite[§2]{tuomov2024tracking}, also applying the triangle and Young's inequalities, we obtain
    \begin{equation}
        \begin{split}
            \dualprod{E'_k(\predicted{x}{k}) - E'_k(\realoptx^k)}{x - \realoptx^k}_{X_k^*,X_k}
             &
            =
            -
            \dualprod{E'_k(\realoptx^k) - E'_k(\predicted{x}{k})}{x - \realoptx^k}_{X_k^*,X_k}
            \\
             &
            \ge
            -
            \norm{E'_k(\predicted{x}{k}) - E'_k(\realoptx^k)}_{k,*}
            \norm{x - \realoptx^k}_{k,\circ}
            \\
             &
            \ge
            -
            L_k'
            \norm{\predicted{x}{k} - \realoptx^k}_{k,\circ}
            \norm{x - \realoptx^k}_{k,\circ}
            \\
             &
            \ge
            -
            L_k'
            \left(
            \norm{x - \predicted{x}{k}}_{k,\circ}
            +
            \norm{x - \realoptx^k}_{k,\circ}
            \right)
            \norm{x - \realoptx^k}_{k,\circ}
            \\
             &
            =
            -
            \frac{L_k'}{2\tilde\gamma}\norm{x - \predicted{x}{k}}_{k,\circ}^2
            -
            L_k'\left(1+\frac{\tilde\gamma}{2}\right)\norm{x - \realoptx^k}_{k,\circ}^2.
        \end{split}
    \end{equation}
    Hence,
    \[
        \begin{aligned}
            \dualprod{\estdiff{E}{k}(\predicted{x}{k}) - E'_k(\realoptx^k)}{x - \realoptx^k}_{X_k^*,X_k}
             &
            =
            \dualprod{\estdiff{E}{k}(\predicted{x}{k}) - E'_k(\predicted{x}{k})}{x - \realoptx^k}_{X_k^*,X_k}
            +
            \dualprod{E'_k(\predicted{x}{k}) - E'_k(\realoptx^k)}{x - \realoptx^k}_{X_k^*,X_k}
            \\
             &
            \ge
            \dualprod{\estdiff{E}{k}(\predicted{x}{k}) - E'_k(\predicted{x}{k})}{x - \realoptx^k}_{X_k^*,X_k}
            \\
            \MoveEqLeft[-1]
            -
            \frac{L_k'}{2\tilde\gamma}\norm{x - \predicted{x}{k}}_{k,\circ}^2
            -
            L_k'\left(1+\frac{\tilde\gamma}{2}\right)\norm{x - \realoptx^k}_{k,\circ}^2.
        \end{aligned}
    \]
    Combining this with \eqref{eq:err-bounds-i} applied to $x = x^k$ and $z=\realoptx^k$ gives \eqref{eq:err-iii-final} for the stated coefficients.

    \cref{item:tracking:smoothness-verification:local}:
    To prove \cref{eq:err-ii-final}, let $\bar x^{1:k} \in \PpredictConstr_{1:k} \isect \prod_{n=1}^N B(\predicted{x}{n},\delta)$, and $x^k \in B(\bar x^k,\delta)$.
    Write $\dualprod{\estdiff{E}{k}(\predicted{x}{k})}{x^k - \bar x^k}_{X_k^*,X_k} = \dualprod{E'_k(\predicted{x}{k})}{x^k - \bar x^k}_{X_k^*,X_k} + \dualprod{\estdiff{E}{k}(\predicted{x}{k}) - E'_k(\predicted{x}{k})}{x^k - \bar x^k}_{X_k^*,X_k}$, apply \eqref{eq:err-bounds-i} with $x = x^k$ and $z = \bar x^k$ and combine it with \eqref{eq:true-3pt-ii-hat-ii}. This gives \eqref{eq:err-ii-final}. Similarly, taking $x = x^k$ and $z = \realoptx^k$ in \eqref{eq:err-bounds-i} and combining it with \eqref{eq:true-3pt-i-hat-ii} gives \eqref{eq:err-i-final}, finishing the proof.
\end{proof}

\subsection{Quadratic energies}
\label{sec:online-tracking:smoothness-quadratic}

We prove \cref{eq:true-3pt-ii-hat-ii,eq:true-3pt-i-hat-ii} for $E_k$ of the quadratic form \eqref{eq:eit:functionals}, relevant in applications with Gaussian noise.
The arguments rely only on local bounds on the Hessian, combined with with global boundedness of $S_{k}$ and a global Lipschitz bound on $S_k'$.

\begin{assumption}[Local strong convexity and smoothness]
    \label{ass:local-around-hat-final}
    Fix $k \in \N$. Let $X_k,Y_k$ be real Hilbert spaces, let $\norm{\cdot}_{k, \circ}$ be a seminorm on $X_k$, and let  $\hat x \in X_k$ and $S_{k,u}:X_k\to Y_k$ be $C^2$ on $B(\hat x,\delta)$ for some $\delta > 0$.
    Let $B_k:Y_k\to Y_k$ be bounded, self-adjoint, and positive semidefinite, and define
    \[
        \norm{y}_{Y_k,\circ}^2 \defeq \iprod{B_k y}{y}_{Y_k}
        \quad\text{for all}\quad y\in Y_k.
    \]
    Define
    \[
        E_k(x) \defeq \frac12\norm{S_{k,u}(x)}_{Y_k,\circ}^2.
    \]
    Assume that $S_{k,u}$ is Fréchet differentiable on $X_k$ and that there exist constants $\bar L_{S,k}\ge0$, $\bar M_k\ge0$, and $\bar r_k\ge0$ such that
    \[
        \norm{S_{k,u}(x)}_{Y_k,\circ}
        \le
        \bar r_k
        \quad\text{for all } x\in X_k,
    \]
    \[
        \norm{S_{k,u}'(x)h}_{Y_k,\circ}
        \le
        \bar L_{S,k}\norm{h}_{k,\circ}
        \quad\text{for all } x\in X_k,\ h\in X_k,
    \]
    and
    \[
        \norm{(S_{k,u}'(u)-S_{k,u}'(v))h}_{Y_k,\circ}
        \le
        \bar M_k\norm{u-v}_{k,\circ}\norm{h}_{k,\circ}
        \quad\text{for all } u,v\in X_k,\ h\in X_k.
    \]
    Define
    \[
        L'_k\defeq \bar L_{S,k}^2+\bar M_k\bar r_k.
    \]    Assume there exist constants $\mu_k>0$, $L_{S,k}\ge0$, $M_k\ge0$, $r_k\ge0$ such that for all
    $x\in B(\hat x,\delta)$:
    \begin{subequations}
        \begin{align}
            \label{eq:GNf-der-lb}
            \norm{S_{k,u}'(x)h}_{Y_k,\circ}
             & \ge \mu_k\norm{h}_{k,\circ}
            \quad\text{for all}\quad h\in X_k,
            \\
            \label{eq:GNf-der-bdd}
            \norm{S_{k,u}'(x)h}_{Y_k,\circ}
             & \le L_{S,k}\norm{h}_{k,\circ}
            \quad\text{for all}\quad h\in X_k,
            \\
            \label{eq:GNf-res-bdd}
            \norm{S_{k,u}(x)}_{Y_k,\circ}
             & \le r_k,
            \\
            \label{eq:GNf-sec-bdd}
            \norm{S_{k,u}''(x)[h_1,h_2]}_{Y_k,\circ}
             & \le M_k\norm{h_1}_{k,\circ}\norm{h_2}_{k,\circ}
            \quad\text{for all}\quad h_1,h_2\in X_k.
        \end{align}
    \end{subequations}
    Define
    \[
        m_k\defeq \mu_k^2-M_k r_k,\quad L''_k\defeq L_{S,k}^2+M_k r_k,
    \]
    and suppose that $m_k > 0$.

\end{assumption}

\begin{corollary}[Global Lipschitz continuity of $E_k'$]
    \label{cor:Ek-global-lipschitz}
    Suppose that \cref{ass:local-around-hat-final} holds. Then $E_k$ is Fréchet differentiable on $X_k$ and
    \[
        \norm{E_k'(u)-E_k'(v)}_{k,*}
        \le
        L'_k\norm{u-v}_{k,\circ}
        \quad\text{for all } u,v\in X_k.
    \]
\end{corollary}

\begin{proof}Since $y\mapsto \frac12\iprod{B_k y}{y}_{Y_k}$ is Fréchet differentiable and $S_{k,u}$ is Fréchet differentiable on $X_k$, the chain rule gives
    \[
        E_k'(x)h
        =
        \iprod{B_kS_{k,u}(x)}{S_{k,u}'(x)h}_{Y_k}.
    \]
    Let $u,v,h\in X_k$. By \cref{ass:local-around-hat-final},
    \[
        \norm{S_{k,u}(u)-S_{k,u}(v)}_{Y_k,\circ}
        \le
        \bar L_{S,k}\norm{u-v}_{k,\circ}.
    \]
    Therefore
    \[
    \begin{aligned}
        \abs{(E_k'(u)-E_k'(v))h}
        &
        \le
        \norm{S_{k,u}(u)-S_{k,u}(v)}_{Y_k,\circ}
        \norm{S_{k,u}'(u)h}_{Y_k,\circ}
        +
        \norm{S_{k,u}(v)}_{Y_k,\circ}
        \norm{(S_{k,u}'(u)-S_{k,u}'(v))h}_{Y_k,\circ}
        \\
        &
        \le
        \left(\bar L_{S,k}^2+\bar M_k\bar r_k\right)
        \norm{u-v}_{k,\circ}\norm{h}_{k,\circ}
        =
        L'_k\norm{u-v}_{k,\circ}\norm{h}_{k,\circ}.
    \end{aligned}
    \]
    Taking the supremum over $\norm{h}_{k,\circ}\le 1$ gives the claim.
\end{proof}

\begin{remark}[Twice continous differentiability of the EIT solution operator]
    The assumption that $S_{k,u}$ be $C^2$ from the Hilbert space $X_k$ to $Y_k$ is somewhat strict.
    For $S_{k,u}=(I_i^{j,k})_{i=1,\ldots,\Nelec;\,j=1,\ldots,\Nmeas}$ the solution operator of the EIT problem \eqref{eq:eit:cem}, with $Y_k=\R^{\Nmeas\Nelec}$, in \cite{tuomov2024online-eit} we prove that $S_{k,u}$ is $C^2$ on $L^\infty$.
    Hence, if $X_k$ is chosen so that $X_k \hookrightarrow L^\infty$ continuously, for instance $X_k=H^s(\Omega)$ with $s>d/2$, then the restriction of $S_{k,u}$ to $X_k$ is Fréchet differentiable as a map from $X_k$ to $Y_k$. Moreover, the $L^\infty$-based derivative estimates imply corresponding $X_k$-based estimates through the embedding $X_k \hookrightarrow L^\infty$. In particular, a uniform bound for the second derivative gives the Lipschitz estimate for $S_{k,u}'$ required above.
\end{remark}

The next lemma shows that local $\norm{\freevar}_\circ$-strong convexity is induced by a uniform lower bound on the derivative $S_{k,u}'$, provided the residual $\norm{S_{k,u}(x)}_{Y_k,\circ}$ is sufficiently small, i.e. if $m_k > 0$. This smallness condition corresponds to working near a critical point with small misfit.\begin{lemma}[Local strong convexity and smoothness for quadratic $E_k$]
    \label{lem:hess-final}
    Suppose that \cref{ass:local-around-hat-final} holds.
    Then for all  $x\in B(\hat x,\delta)$ and $h\in X_k$,
    \begin{equation}
        \label{eq:Ek-local-strong-convexity}
        m_k\norm{h}_{k,\circ}^2
        \le
        E''_k(x)[h,h]
        \le
        L''_k\norm{h}_{k,\circ}^2.
    \end{equation}
\end{lemma}\begin{proof}Let $x\in B(\hat x,\delta)$. By the definition of $E_k$ and the $C^2$-regularity of $S_{k,u}$ on $B(\hat x,\delta)$, $E_k$ is $C^2$ on $B(\hat x,\delta)$ and
    \[
        E'_k(x)h
        =
        \iprod{B_kS_{k,u}(x)}{S_{k,u}'(x)h}_{Y_k},
        \quad
        \text{and}
        \quad
        E''_k(x)[h,h]
        =
        \iprod{B_kS_{k,u}'(x)h}{S_{k,u}'(x)h}_{Y_k}
        +
        R_k(x)[h,h],
    \]
    where
    $
        R_k(x)[h,h]
        \defeq
        \iprod{B_kS_{k,u}(x)}{S_{k,u}''(x)[h,h]}_{Y_k}.
    $
    By the Cauchy--Schwarz inequality for the positive semidefinite bilinear form $(u,v) \mapsto \iprod{B_k u}{v}_{Y_k}$,
    together with \eqref{eq:GNf-res-bdd} and \eqref{eq:GNf-sec-bdd}, we have
    \[
        \abs{R_k(x)[h,h]}
        \le
        \norm{S_{k,u}(x)}_{Y_k,\circ}
        \norm{S_{k,u}''(x)[h,h]}_{Y_k,\circ}
        \le
        M_k r_k \norm{h}_{k,\circ}^2.
    \]
    Also, by the definition of $\norm{\cdot}_{Y_k,\circ}$ and
    \eqref{eq:GNf-der-lb}--\eqref{eq:GNf-der-bdd},
    \[
        \mu_k^2\norm{h}_{k,\circ}^2
        \le
        \norm{S_{k,u}'(x)h}_{Y_k,\circ}^2
        \le
        L_{S,k}^2\norm{h}_{k,\circ}^2.
    \]
    Summing yields the claim.
\end{proof}

The following lemma provides the standard upper and lower quadratic models for $E_k$. These inequalities form the basic building blocks for the three-point estimates.\begin{lemma}[Descent and support inequalities]
    \label{lem:onestep-final}
    Suppose that \cref{ass:local-around-hat-final} holds for some $\hat x \in X_k$. Then for all $x,z \in B(\hat x,\delta) \subset X_k$,
    \begin{align}
        \label{eq:descent-final}
        E_k(x)
         & \le
        E_k(z)
        +
        \dualprod{E_k'(z)}{x - z}_{X_k^*,X_k}
        +
        \frac{L''_k}{2}\norm{x - z}_{k,\circ}^2,
        \\
        \label{eq:support-final}
        E_k(x)
         & \ge
        E_k(z)
        +
        \dualprod{E_k'(z)}{x - z}_{X_k^*,X_k}
        +
        \frac{m_k}{2}\norm{x - z}_{k,\circ}^2.
    \end{align}
\end{lemma}\begin{proof}Let $x,z \in B(\hat x,\delta)$ and set $\phi(t) = E_k(z + t(x - z))$.
    Since $B(\hat x,\delta)$ is convex, $z+t(x-z)\in B(\hat x,\delta)$ for all $t\in[0,1]$ and thus, since \cref{ass:local-around-hat-final} holds, by \cref{lem:hess-final}, $\phi(t)$ is $C^2(0,1)$, $\phi'(t) = \dualprod{E_k'(z + t(x-z))}{x-z}_{X_k^*,X_k}$, and $\phi''(t) = E_k''(z+t(x-z))[x-z,x-z]$. Moreover, \eqref{eq:Ek-local-strong-convexity} holds for every $t$ and gives $m_k\norm{x-z}_{k,\circ}^2\le \phi''(t)\le L''_k\norm{x-z}_{k,\circ}^2$. Then
    \[
        \begin{aligned}
            \frac{m_k}{2}\norm{x - z}_{k,\circ}^2
            =
            \int_0^1\int_0^s m_k\norm{x - z}_{k,\circ}^2 dt ds
             &
            \le
            \int_0^1\int_0^s \phi''(t) dt ds
            \\
             &
            \le
            \int_0^1\int_0^s L''_k\norm{x - z}_{k,\circ}^2 dt ds
            =
            \frac{L''_k}{2}\norm{x-z}_{k,\circ}^2
        \end{aligned}
    \]
    Since $\int_0^1\int_0^s \phi''(t) dt ds = \phi(1) - \phi(0) - \phi'(0) = E_k(x) - E_k(z) - \dualprod{E_k'(z)}{x-z}_{X_k^*,X_k}$, \eqref{eq:descent-final} -- \eqref{eq:support-final} follow.
\end{proof}

The following lemma shows that combining the local support inequalities with the global Lipschitz bound on $E_k'$ yields a family of three-point inequalities. These inequalities quantify how well the linearization at a possibly nonlocal point $z$ approximates the energy difference between nearby points. The estimates allow for a trade-off between the curvature term and the error $\norm{x - z}_{k,\circ}^2$ controlled by free parameters.\begin{lemma}[Three-point inequalities]
    \label{lem:true-3pt-hat}
    Suppose that \cref{ass:local-around-hat-final} holds for some $\hat x \in X_k$.
    Then for all $z,x\in X_k$,
    \[
        \dualprod{E_k'(z) - E_k'(\hat x)}{x - z}_{X_k^*,X_k}
        \ge
        - L'_k
        \norm{z - \hat x}_{k,\circ}
        \norm{x - z}_{k,\circ}.
    \]
    Moreover, for all $z\in X_k$, all $x,\bar x\in B(\hat x, \delta)$, and any $\theta>0$,
    \begin{equation}
        \label{eq:true-3pt-ii-hat}
        \dualprod{E_k'(z)}{x - \bar x}_{X_k^*,X_k}
        \ge
        E_k(x) - E_k(\bar x)
        +
        \frac{m_k}{2(1 + \theta)}\norm{x - \bar x}_{k,\circ}^2
        -
        \frac{(L'_k)^2(1 + \theta)}{2m_k\theta}
        \norm{x - z}_{k,\circ}^2.
    \end{equation}
    Moreover, for all $z\in X_k$, all $x\in B(\hat x, \delta)$, and any $\epsilon>0$ such that $\epsilon L'_k < m_k$,
    \begin{equation}
        \label{eq:true-3pt-i-hat}
        \dualprod{E_k'(z) - E_k'(\hat x)}{x - \hat x}_{X_k^*,X_k}
        \ge
        \left(
        m_k - \epsilon L'_k
        \right)
        \norm{x - \hat x}_{k,\circ}^2
        -
        \frac{L'_k}{4\epsilon}
        \norm{x - z}_{k,\circ}^2.
    \end{equation}
\end{lemma}\begin{proof}The first estimate follows from the definition of $\norm{\cdot}_{k,*}$ and the global Lipschitz continuity of $E_k'$, and the Cauchy--Schwarz inequality between $\norm{\cdot}_{k,*}$ and $\norm{\cdot}_{k,\circ}$ (see \cite[§2]{tuomov2024tracking})
    \[
        \begin{aligned}
            \dualprod{E_k'(z) - E_k'(\hat x)}{x - z}_{X_k^*,X_k}
             &
            \ge
            -
            \norm{E_k'(z) - E_k'(\hat x)}_{k,*}
            \norm{x - z}_{k,\circ}
            \ge
            - L'_k
            \norm{z - \hat x}_{k,\circ}
            \norm{x - z}_{k,\circ}.
        \end{aligned}
    \]
    To prove \eqref{eq:true-3pt-ii-hat}, let $z\in X_k$ and $x,\bar x\in B(\hat x,\delta)$. Applying \eqref{eq:support-final} to $(\bar x,x)$ gives
    \[
        \dualprod{E_k'(x)}{x - \bar x}_{X_k^*,X_k}
        \ge
        E_k(x) - E_k(\bar x)
        +
        \frac{m_k}{2}\norm{x - \bar x}_{k,\circ}^2.
    \]
    Hence, by the global Lipschitz continuity of $E_k'$ proven in \cref{cor:Ek-global-lipschitz},
    \[
        \begin{aligned}
            \dualprod{E_k'(z)}{x - \bar x}_{X_k^*,X_k}
             &
            =
            \dualprod{E_k'(x)}{x - \bar x}_{X_k^*,X_k}
            +
            \dualprod{E_k'(z) - E_k'(x)}{x - \bar x}_{X_k^*,X_k}
            \\
             &
            \ge
            E_k(x) - E_k(\bar x)
            +
            \frac{m_k}{2}\norm{x - \bar x}_{k,\circ}^2
            -
            L'_k\norm{x - z}_{k,\circ}\norm{x - \bar x}_{k,\circ}.
        \end{aligned}
    \]
    For any $\theta>0$, Young's inequality gives
    \[
        L'_k\norm{x - z}_{k,\circ}\norm{x - \bar x}_{k,\circ}
        \le
        \frac{m_k\theta}{2(1 + \theta)}\norm{x - \bar x}_{k,\circ}^2
        +
        \frac{(L'_k)^2(1 + \theta)}{2m_k\theta}\norm{x - z}_{k,\circ}^2.
    \]
    Substituting this yields \eqref{eq:true-3pt-ii-hat}.

    To prove \eqref{eq:true-3pt-i-hat}, apply \eqref{eq:support-final} to $(x,\hat x)$ and to $(\hat x,x)$ to obtain
    \[
        \begin{aligned}
            E_k(x) - E_k(\hat x) - \dualprod{E_k'(\hat x)}{x - \hat x}_{X_k^*,X_k}
             &
            \ge
            \frac{m_k}{2}\norm{x - \hat x}_{k,\circ}^2,
            \\
            E_k(\hat x) - E_k(x) - \dualprod{E_k'(x)}{\hat x - x}_{X_k^*,X_k}
             &
            \ge
            \frac{m_k}{2}\norm{x - \hat x}_{k,\circ}^2.
        \end{aligned}
    \]
    Summing these gives
    \[
        \dualprod{E_k'(x) - E_k'(\hat x)}{x - \hat x}_{X_k^*,X_k}
        \ge
        m_k\norm{x - \hat x}_{k,\circ}^2.
    \]
    Therefore, by the global Lipschitz continuity of $E_k'$,
    \[
        \begin{aligned}
            \dualprod{E_k'(z) - E_k'(\hat x)}{x-\hat x}_{X_k^*,X_k}
             &
            =
            \dualprod{E_k'(x) - E_k'(\hat x)}{x-\hat x}_{X_k^*,X_k}
            +
            \dualprod{E_k'(z) - E_k'(x)}{x-\hat x}_{X_k^*,X_k}
            \\
             &
            \ge
            m_k\norm{x - \hat x}_{k,\circ}^2
            -
            L'_k\norm{x - z}_{k,\circ}\norm{x-\hat x}_{k,\circ}.
        \end{aligned}
    \]
    Young's inequality gives
    \[
        L'_k\norm{x - z}_{k,\circ}\norm{x - \hat x}_{k,\circ}
        \le
        \epsilon L'_k\norm{x - \hat x}_{k,\circ}^2
        +
        \frac{L'_k}{4\epsilon}\norm{x - z}_{k,\circ}^2.
    \]
    Substituting this yields \eqref{eq:true-3pt-i-hat}.
\end{proof}

We conclude with a corollary that explicitly proves the assumptions of \cref{thm:tracking:smoothness-verification} regarding $E_k'$.
Thus combined with \cref{thm:tracking:smoothness-verification} , the next result completes the verification of \cref{ass:pd:smoothness} for the estimates $\thisestdiff$.

\begin{corollary}[Instantiation of three-point inequalities]
    \label{cor:3pt-to-ass-smoothness}
    Let $\delta > 0$ and suppose that \cref{ass:local-around-hat-final} holds with centre $\realoptx^k$ for any $\this\realoptu = (\this{\realoptx}, \this{\realopty}) \in H_k^{-1}(0)$, and with centre $\bar x^k$ for the $k$th component $\bar x^k$ of every comparison sequence $\bar x^{1:k} \in \PpredictConstr_{1:k}$ satisfying $\bar x^k \in B(\predicted{x}{k}, \delta)$, with the same constants $m_k$ and $L'_k$.
    Then $E'_k$ is globally Lipschitz continuous and the three-point inequalities \cref{eq:true-3pt-ii-hat-ii,eq:true-3pt-i-hat-ii} from \cref{thm:tracking:smoothness-verification} hold on $B(\realoptx^k, \delta)$ and on $B(\bar x^k, \delta)$ with the following explicit constants:
    \begin{enumerate}[label=(\roman*)]
        \item\label{itm:inst-tpi:i}
              For any $\theta > 0$, for every comparison sequence satisfying $\bar x^{1:N} \in \PpredictConstr_{1:N} \cap \prod_{i=1}^N B(\predicted{x}{i},\delta)$, for all $x \in B(\bar x^k, \delta)$ and all $z \in X_k$,
              \[
                  \dualprod{E_k'(z)}{x - \bar x^k}_{X_k^*,X_k}
                  \ge
                  E_k(x) - E_k(\bar x^k)
                  +
                  \frac{m_k}{2(1 + \theta)}\norm{x - \bar x^k}_{k,\circ}^2
                  -
                  \frac{(L'_k)^2(1 + \theta)}{2m_k\theta}\norm{x - z}_{k,\circ}^2.
              \]
              In particular, if $x^k \in B(\bar x^k, \delta)$, choosing $x = x^k$ and $z = \predicted{x}{k}$ yields \eqref{eq:true-3pt-ii-hat-ii} with
              \begin{equation}
                  \label{eq:c1c2-choice}
                  \gamma_1 \defeq \frac{m_k}{2(1 + \theta)},
                  \quad
                  L_1 \defeq \frac{(L'_k)^2(1 + \theta)}{2m_k\theta},
              \end{equation}
              i.e.
              \[
                  \dualprod{E_k'(\predicted{x}{k})}{x^k - \bar x^k}_{X_k^*,X_k}
                  \ge
                  E_k(x^k) - E_k(\bar x^k)
                  +
                  \gamma_1\norm{x^k - \bar x^k}_{k,\circ}^2
                  -
                  L_1\norm{x^k - \predicted{x}{k}}_{k,\circ}^2.
              \]

        \item\label{itm:inst-tpi:ii}
              Let $\epsilon > 0$ satisfy $\epsilon L'_k < m_k$. Then for any $\this\realoptu = (\this{\realoptx}, \this{\realopty}) \in H_k^{-1}(0)$, for all $x \in B(\realoptx^k, \delta)$ and all $z \in X_k$,
              \[
                  \dualprod{E_k'(z) - E_k'(\realoptx^k)}{x - \realoptx^k}_{X_k^*,X_k}
                  \ge
                  \left(
                  m_k - \epsilon L'_k
                  \right)
                  \norm{x - \realoptx^k}_{k,\circ}^2
                  -
                  \frac{L'_k}{4\epsilon}
                  \norm{x - z}_{k,\circ}^2.
              \]
              In particular, if $x^k \in B(\realoptx^k, \delta)$, choosing $x = x^k$ and $z = \predicted{x}{k}$ yields \eqref{eq:true-3pt-i-hat-ii} with
              \begin{equation}
                  \label{eq:c3c4-choice}
                  \gamma_2 \defeq m_k - \epsilon L'_k,
                  \quad
                  L_2 \defeq \frac{L'_k}{4\epsilon},
              \end{equation}
              i.e.
              \[
                  \dualprod{E_k'(\predicted{x}{k})
                      -
                      E_k'(\realoptx^k)}{x^k - \realoptx^k}_{X_k^*,X_k}
                  \ge
                  \gamma_2\norm{x^k - \realoptx^k}_{k,\circ}^2
                  -
                  L_2\norm{x^k - \predicted{x}{k}}_{k,\circ}^2.
              \]
    \end{enumerate}
\end{corollary}

\begin{proof}The Lipschitz continuity is proven by \cref{cor:Ek-global-lipschitz}. Next, apply \cref{lem:true-3pt-hat}.
    For \cref{itm:inst-tpi:i}, fix a comparison sequence $\bar x^{1:k} \in \PpredictConstr_{1:k}$ whose $k$th component $\bar x^k$ satisfies $\bar x^k \in B(\predicted{x}{k}, \delta)$. Apply \cref{lem:true-3pt-hat} with centre $\bar x^k$, take $z = \predicted{x}{k}$ and $\bar x = \bar x^k$, and read off $\gamma_1$ and $L_1$ as in \eqref{eq:c1c2-choice}.
    For \cref{itm:inst-tpi:ii}, fix any $\this\realoptu = (\this{\realoptx}, \this{\realopty}) \in H_k^{-1}(0)$. Apply \cref{lem:true-3pt-hat} with centre $\realoptx^k$, take $z = \predicted{x}{k}$, and read off $\gamma_2$ and $L_2$ as in \eqref{eq:c3c4-choice}.
\end{proof}

\section{Construction of online inner and adjoint algorithms}
\label{sec:inner-adjoint}

In this section, we present inner and adjoint algorithms that satisfy \cref{ass:tracking:main}.
Let $U_k$ and $W_k$ be normed spaces, and let
\[
    d_{U_k}(u, \tilde u) = \norm{u-\tilde u}_{U_k}
    \quad\text{and}\quad
    d_{W_k}(w, \tilde w) = \norm{w-\tilde w}_{W_k}.
\]
To prove \cref{ass:tracking:main}\,\cref{item:tracking:main:inner-tracking,item:tracking:main:adjoint-tracking}, our strategy is to combine contractivity properties of the algorithms with stability properties of the predictors and solution mappings. More precisely:
\begin{enumerate}\item \textbf{Bounded Lipschitz predictors:}
          We i the outer primal-dual, inner, and adjoint predictor mappings
          \[
              P_k : \Omega_{k-1} \to \Omega_k, \quad
              Q_k : U_{k-1} \to U_k, \quad
              \tilde{Q}_k : W_{k-1} \to W_k.
          \]
          Then
          \[
              \predicted{x}{k} \defeq P_k(x^{k-1}), \quad
              \predicted{u}{k} \defeq Q_k(u^{k-1}), \quad
              \predicted{w}{k} \defeq \tilde{Q}_k(w^{k-1}).
          \]
          These predictors describe how iterates and auxiliary variables are transported between successive time steps and data frames.
          We will generally assume these predictors to be bounded and Lipschitz. More precisely, the predictor $Q_k$ is $L_Q$-Lipschitz:
          \begin{equation}
              \label{eq:inner-adjoint:inner-pred-lipschitz}
              \norm{Q_k(u) - Q_k(v)}_{U_{k}} \le L_Q \norm{u - v}_{U_{k-1}},\quad\text{for all}\quad u,v \in U_{k-1}.
          \end{equation}
          Also the adjoint predictor $\tilde Q_k: W_k \to W_{k+1}$ be $L_{\tilde Q}$-Lipschitz, i.e.,
          \begin{equation}
              \label{eq:inner-adjoint:adjoint-pred-lipschitz}
              \norm{\tilde Q_k(w_1) - \tilde Q_k(w_2)}_{W_{k}} \leq L_{\tilde Q} \norm{w_1 - w_2}_{W_{k-1}}
              \quad\text{for all}\quad w_1, w_2 \in W_{k-1}.
          \end{equation}

    \item \textbf{Lipschitz solution mappings:}
          The inner problem solution mapping $S_{k,u}$ is $L_{S_u}$-Lipschitz on $\Omega_k$, i.e.,
          \begin{equation}
              \label{eq:inner-adjoint:inner-sol-lipschitz}
              \norm{S_{k,u}(x) - S_{k,u}(\tilde{x})}_{U_k} \le L_{S_u} \norm{x - \tilde{x}}_{k,\circ},\quad\text{for all}\quad x, \tilde{x} \in \Omega_k.
          \end{equation}
          Likewise, the adjoint solution mapping $S_{k,w}$ is $L_{S_w}$-Lipschitz on $\Omega_k$, i.e.,
          \begin{equation}
              \label{eq:inner-adjoint:adjoint-sol-lipschitz}
              \norm{S_{k,w}(x) - S_{k,w}(\tilde x)}_{W_k} \leq L_{S_w} \norm{x - \tilde x}_{k,\circ}
              \quad\text{for all}\quad x, \tilde x \in \Omega_k,
          \end{equation}
          These conditions quantify the stability of the inner and adjoint solution maps with respect to perturbations in the parameter $x$.

        \item \textbf{Non-parametric inner contractivity:} The inner iterate $\nextu$ is computed from $\thisu$ by single or multiple steps of a standard optimisation or linear system splitting algorithm. This algorithm satisfies, for some $\bar \kappa_u >L_Q$, the inequality
        \begin{equation}
            \label{eq:inner-adjoint:inner-contr}
            \bar \kappa_u \norm{\nextu - S_{k+1,u}(\predicted{x}{k+1})}_{U_{k+1}}
            \le
            \norm{\predicted{u}{k+1} - S_{k+1,u}(\predicted{x}{k+1})}_{U_{k+1}}.
        \end{equation}
        This ensures that each inner iteration advances towards the inner solution $S_{k+1,u}(\predicted{x}{k+1})$ for the parameter $\predicted{x}{k+1}$, and the current data embedded in the inner objective.
        Typically $L_Q>1$, which makes \eqref{eq:inner-adjoint:inner-contr} a contractivity condition.

        \item \textbf{Non-parametric adjoint contractivity:} Similarly, the adjoint variable $\nexxt w$ is computed from $\this w$ by an algorithm.
        This algorithm satisfies for some $\bar \kappa_w > L_{\tilde Q}$ and $L_w \geq 0$ the inequality
        \begin{equation}
            \label{eq:inner-adjoint:adjoint-contr}
            \bar \kappa_w \norm{ w^{k+1} - S_{k+1,w}(\predicted{x}{k+1})}_{W_{k+1}}
            \leq
            \norm{\predicted{w}{k+1} - S_{k+1,w}(\predicted{x}{k+1})}_{W_{k+1}}
            + L_w \norm{u^{k+1} - S_{k+1,u}(\predicted{x}{k+1})}_{U_{k+1}}.
        \end{equation}
        Typically $L_{\tilde Q}>1$, which makes \eqref{eq:inner-adjoint:inner-contr} a contractivity condition with an additional term that allows for inexact inner solutions.
\end{enumerate}
We proceed with this plan by deriving general results in \cref{sec:inner-adjoint:inner,sec:inner-adjoint:adjoint}, for inner and adjoint algorithms, respectively.
In \cref{sec:inner-adjoint:differential-transformation}, we then treat the differential transformation \cref{ass:tracking:main}\,\cref{item:tracking:main:differential-transformation}.
Finally, in \cref{sec:inner-adjoint:linear} we concentrate on the specific case linear system splitting methods for the inner and adjoint PDE, relevant to our EIT application.

\subsection{Composed estimates for inner algorithms}
\label{sec:inner-adjoint:inner}

The following lemma proves \cref{ass:tracking:main}\,\cref{item:tracking:main:inner-tracking} subject to the errors $\innerError{k}$ still having an unquantified form, which will depend on properties of the predictor $Q_k$ that we analyse afterwards.\begin{lemma}\label{lemma:inner-adjoint:inner-simple-contractive}
    For all $k\ge 1$, assume that \cref{eq:inner-adjoint:inner-pred-lipschitz} (Lipschitz inner predictor $Q_k$) and \eqref{eq:inner-adjoint:inner-sol-lipschitz} (Lipschitz inner solution mapping $S_{k,u}$) hold.
    Also suppose we are given an inner algorithm that, given
    $
        \big((x^n, \predicted{x}{n}, u^n)\big)_{n=0}^k \in \prod_{n=0}^k \localset_n \times \localset_n \times U_n
    $
    produces $\nextu \in U_{k+1}$ satisfying for a constant $\bar \kappa_u>L_Q$ the basic contractivity \eqref{eq:inner-adjoint:inner-contr}.
    Then \cref{ass:tracking:main}\,\cref{item:tracking:main:inner-tracking} holds, more precisely
    \begin{gather*}
        \begin{split}
            \frac{\bar \kappa_u}{L_Q} \norm{u^{k+1} - S_{k+1,u}(\predicted{x}{k+1})}_{U_{k+1}}
             &
            \leq
            \norm{u^k - S_{k,u}(\predicted{x}{k})}_{U_k} + L_{S_u} \norm{\predicted{x}{k} - \thisx}_{k,\circ} + \innerError{k}
        \end{split}
        \shortintertext{for}
        \innerError{k}
        \defeq
        \inv L_Q\norm{Q_{k+1}(S_{k,u}(\thisx)) - S_{k+1,u}(P_{k+1}(\thisx))}_{U_{k+1}}.
    \end{gather*}
\end{lemma}

\begin{proof}Using the definition of $\breve u^{k+1}$, we get
    \[
        \breve u^{k+1} - S_{k+1,u}(\predicted{x}{k+1})
        =
        [Q_{k+1}(u^k) - Q_{k+1}(S_{k,u}(\predicted{x}{k}))] + [Q_{k+1}(S_{k,u}(\predicted{x}{k})) - S_{k+1,u}(\predicted{x}{k+1})].
    \]
    Then
    \begin{equation}
        \label{eq:inner-adjoint:inner-simple-contractive:2}
        \begin{split}
            \norm{\breve u^{k+1} - S_{k+1,u}(\predicted{x}{k+1})}_{U_{k+1}}
             &
            \leq
            \norm{Q_{k+1}(u^k) - Q_{k+1}(S_{k,u}(\predicted{x}{k}))}_{U_{k+1}}
            \\
            \MoveEqLeft[-1]
            + \norm{Q_{k+1}(S_{k,u}(\predicted{x}{k})) - S_{k+1,u}(\predicted{x}{k+1})}_{U_{k+1}}
            \\
             &
            \le
            L_Q \norm{u^k - S_{k,u}(\predicted{x}{k})}_{U_k}
            + \norm{Q_{k+1}(S_{k,u}(\predicted{x}{k})) - Q_{k+1}(S_{k,u}(\thisx))}_{U_{k+1}}
            \\
            \MoveEqLeft[-1]
            + \norm{Q_{k+1}(S_{k,u}(\thisx)) - S_{k+1,u}(\predicted{x}{k+1})}_{U_{k+1}}
            \\
             &
            \le
            L_Q \norm{u^k - S_{k,u}(\predicted{x}{k})}_{U_k}
            +  L_Q L_{S_u}\norm{\predicted{x}{k} - \thisx}_{k,\circ}
            \\
            \MoveEqLeft[-1]
            + \norm{Q_{k+1}(S_{k,u}(\thisx)) - S_{k+1,u}(\predicted{x}{k+1})}_{U_{k+1}}.
        \end{split}
    \end{equation}
    Combining \cref{eq:inner-adjoint:inner-contr,eq:inner-adjoint:inner-simple-contractive:2}, and using $ P_{k+1}(x^k) = \predicted{x}{k+1}$ as well as $\bar \kappa_u>L_Q$, we obtain the claim.
\end{proof}

\begin{remark}[Inner algorithms]
    Several standard algorithms satisfy the contractivity assumption \eqref{eq:inner-adjoint:inner-contr} under appropriate conditions.
    As our focus is on PDEs, we treat linear system splitting in \cref{sec:inner-adjoint:linear}.
    For inner optimisation problems, forward-backward splitting is treated in \cite{tuomov2024tracking}.
\end{remark}

\begin{remark}[Remainder term]
    \label{rem:iner-adjoint:sol-map-data-error}
    The error term $\innerError{k}$ is a commutativity error between the solution mappings and the predictions of the parameter and the solutions.
    For \cref{cor:pd:main} we want this term to be at least bounded, preferably with a bounded sum over $k \in \N$.
    For identity predictions when $X_{k+1}=X_k$ and $U_{k+1}=U_k$,
    $
        \innerError{k}
        =
        \norm{S_{k,u}(\thisx) - S_{k+1,u}(\thisx)}_{U_{k+1}},
    $
    so the term depends on the continuity properties of the solution mapping with respect to the embedded data, encoded in the index $k$.

    More generally, for a conservative estimate when $X_k=X_{k+1}$ and $U_k=U_{k+1}$, if $S_{k,u}$ is $L_{S,u}$-Lipschitz for each $k$, we have
    \[
        \begin{split}
            \innerError{k}
             &
            =
            \inv L_Q\norm{[Q_{k+1}-\Id](S_{k,u}(\thisx))+S_{k,u}(\thisx) - S_{k+1,u}(P_{k+1}(\thisx))}_{U_{k+1}}
            \\
             &
            \le
            \inv L_Q\norm{[Q_{k+1}-\Id](S_{k,u}(\thisx))}_{U_{k+1}}
            +\inv L_Q\norm{S_{k,u}(\thisx) - S_{k+1,u}(\thisx)}_{U_{k+1}}
            +\frac{L_{S,u}}{L_Q}\norm{(P_{k+1}-\Id)(\thisx)}_{k+1,\circ}.
        \end{split}
    \]
    The first and last term can be controlled and even made small if $P_{k+1} \approx \Id$, $Q_{k+1} \approx \Id$, and the solution mappings and iterates are bounded.
    Regarding the middle term, for parametrised linear systems $T_k(x,u) = A_{k,x}u-b_{k,x}$, we have $S_{k,u}(x)=\inv A_{k,x}b_{k,x}$, so can further estimate
    \[
        \begin{split}
            \norm{S_{k,u}(\thisx) - S_{k+1,u}(\thisx)}
             &
            =
            \norm{\inv A_{k,\thisx}b_{k,\thisx}-\inv A_{k+1,\thisx}b_{k+1,\thisx}}
            \\
             &
            =
            \norm{(\inv A_{k,\thisx}-\inv A_{k+1,\thisx})b_{k,\thisx}-\inv A_{k+1,\thisx}(b_{k+1,\thisx}-b_{k,\thisx)}}
            \\
             &
            \le
            \norm{(\inv A_{k,\thisx}-\inv A_{k+1,\thisx})}\norm{b_{k,\thisx}}
            +\norm{\inv A_{k+1,\thisx}}\norm{b_{k+1,\thisx}-b_{k,\thisx}}.
        \end{split}
    \]
    This can, again, be controlled if the dependence of $\inv A_{k,\thisx}$ and $b_{k,\thisx}$ on the data (encoded in the index $k$) is continuous.
\end{remark}

\subsection{Composed estimates for adjoint algorithms}
\label{sec:inner-adjoint:adjoint}

We now repeat the overall idea of the arguments of the previous subsection for \cref{ass:tracking:main}\,\cref{item:tracking:main:adjoint-tracking}.\begin{lemma}\label{lemma:inner-adjoint:adjoint-simple-contractive}
    For all $k \ge 1$, assume \cref{eq:inner-adjoint:adjoint-pred-lipschitz} (Lipschitz adjoint predictor $\tilde Q_k$) and \cref{eq:inner-adjoint:adjoint-sol-lipschitz} (Lipschitz adjoint solution mapping $S_{k,w}$).
    Also suppose that we are given an adjoint algorithm that, given $(x^n,\predicted{x}{n},u^n)_{n=0}^k \subset \prod_{n=0}^k \localset_n \times \localset_n\times U_{n}$ and $(x^{k+1}, u^{k+1}) \in \localset_{k+1} \times U_{k+1}$, produces a $\nexxt w \in W_{k+1}$ that satisfies the basic adjoint contractivity condition \eqref{eq:inner-adjoint:adjoint-contr} with constants $\bar \kappa_w > L_{\tilde Q}$ and $L_w \geq 0$.
    Then \cref{ass:tracking:main}\,\cref{item:tracking:main:adjoint-tracking} holds, more precisely
    \begin{gather*}
        \begin{split}
            \frac{\bar \kappa_w}{L_{\tilde Q}} \norm{w^{k+1} - S_{k+1,w}(\predicted{x}{k+1})}_{W_{k+1}}
             &
            \le
            \norm{w^k - S_{k,w}(\predicted{x}{k})}_{W_k} + \frac{L_w}{L_{\tilde Q}} \norm{u^{k+1} - S_{k+1,u}(\predicted{x}{k+1})}_{U_{k+1}}
            \\
            \MoveEqLeft[-1]
            +  L_{S_w}\norm{\predicted{x}{k} - \thisx}_{X_k}  + \adjointError{k}
        \end{split}
        \shortintertext{for}
        \adjointError{k} = \inv L_{\tilde Q} \norm{\tilde Q_{k+1}(S_{k,w}(\thisx)) - S_{k+1,w}(\predicted{x}{k+1})}_{W_{k+1}}.
    \end{gather*}
\end{lemma}\begin{proof}Using the definition of $\predicted{w}{k+1}$, we get
    \[
        \predicted{w}{k+1} - S_{k+1,w}(\predicted{x}{k+1}) = [\tilde Q_{k+1}(w^k) - \tilde Q_{k+1}(S_{k,w}(\predicted{x}{k}))] + [\Tilde{Q}_{k+1}(S_{k,w}(\predicted{x}{k})) - S_{k+1,w}(\predicted{x}{k+1})].
    \]
    Using that $\tilde Q_{k+1}$ is $L_{\tilde Q}$-Lipschitz, we obtain
    \begin{equation}
        \label{eq:tracking:adjoint-simple-contractive:2}
        \begin{split}
            \norm{\predicted{w}{k+1} - S_{k+1,w}(\predicted{x}{k+1})}_{W_{k+1}}
             &
            \leq
            \norm{\Tilde{Q}_{k+1}(w^k) - \Tilde{Q}_{k+1}(S_{k,w}(\predicted{x}{k}))}_{W_{k+1}}
            \\
            \MoveEqLeft[-1]
            + \norm{\Tilde{Q}_{k+1}(S_{k,w}(\predicted{x}{k})) - S_{k+1,w}(\predicted{x}{k+1})}_{W_{k+1}}
            \\
             &
            \leq
            L_{\tilde Q} \norm{w^k - S_{k,w}(\predicted{x}{k})}_{W_k}
            + \norm{\Tilde{Q}_{k+1}(S_{k,w}(\predicted{x}{k})) - S_{k+1,w}(\predicted{x}{k+1})}_{W_{k+1}}
            \\
             &
            \le
            L_{\tilde Q} \norm{w^k - S_{k,w}(\predicted{x}{k})}_{W_k}
            +
            \norm{\Tilde{Q}_{k+1}(S_{k,w}(\predicted{x}{k})) - \Tilde{Q}_{k+1}(S_{k,w}(\thisx))}_{W_{k+1}}
            \\
            \MoveEqLeft[-1]
            + \norm{\Tilde{Q}_{k+1}(S_{k,w}(\thisx)) - S_{k+1,w}(\predicted{x}{k+1})}_{W_{k+1}}
            \\
             &
            \le
            L_{\tilde Q} \norm{w^k - S_{k,w}(\predicted{x}{k})}_{W_k}
            +
            L_{\tilde Q}L_{S_w}\norm{\predicted{x}{k} - \thisx}_{k,\circ}
            \\
            \MoveEqLeft[-1]
            + \norm{\Tilde{Q}_{k+1}(S_{k,w}(\thisx)) - S_{k+1,w}(\predicted{x}{k+1})}_{W_{k+1}},
        \end{split}
    \end{equation}
    where the third inequality follows by adding and subtracting the term $\Tilde{Q}_{k+1}(S_{k,w}(\thisx))$, the fourth inequality follows from the Lipschitz property of $\Tilde{Q}_{k+1}$ and $S_{k,w}$.
    Combining \cref{eq:inner-adjoint:adjoint-contr,eq:tracking:adjoint-simple-contractive:2} and using $\bar \kappa_w > L_{\tilde Q}$, we obtain the claim.
\end{proof}

\begin{remark}[Remainder term]
    \label{rem:tracking:sol-map-data-error-adjoint}
    Similarly to \cref{rem:iner-adjoint:sol-map-data-error}, the error term $\adjointError{k}$ is a commutativity error between the adjoint solution mappings and the predictions of the parameter and the solutions.
    For \cref{cor:pd:main}, we want this term to be at least bounded, preferably with a bounded sum over $k \in \N$.
    For identity predictions when $X_{k+1}=X_k$ and $U_{k+1}=U_k$,
    $
        \adjointError{k}
        =
        \norm{S_{k,w}(\thisx) - S_{k+1,w}(\thisx)}_{W_{k+1}},
    $
    so the term depends on the continuity properties of the solution mapping with respect to the embedded data, encoded in the index $k$.
    More generally, for a conservative estimate when $X_k=X_{k+1}$ and $U_k=U_{k+1}$, if $S_{k,w}$ is $L_{S,w}$-Lipschitz for each $k$, we have
    \[
        \begin{split}
            \adjointError{k}
             &
            =
            \inv L_{\tilde Q}\norm{[\Tilde Q_{k+1}-\Id](S_{k,w}(\thisx))+S_{k,w}(\thisx) - S_{k+1,w}(P_{k+1}(\thisx))}_{W_{k+1}}
            \\
             &
            \le
            \inv L_{\tilde Q}\norm{[\Tilde Q_{k+1}-\Id](S_{k,w}(\thisx))}_{W_{k+1}}
            +\inv L_{\tilde Q}\norm{S_{k,w}(\thisx) - S_{k+1,w}(\thisx)}_{W_{k+1}}
            \\
            \MoveEqLeft[-1]
            +\inv L_{\tilde Q}L_{S,w}\norm{(P_{k+1}-\Id)(\thisx)}_{k+1,\circ}.
        \end{split}
    \]
    The first and last term can be controlled and even made small if $\Tilde{Q}_{k+1} \approx \Id$, $P_{k+1} \approx \Id$, and the solution mappings and iterates are bounded.
    Regarding the middle term, for parametrised linear systems $T_k(x,u) = A_{k,x}u-b_{k,x}$, we have $S_{k,w}(x) = - (A_{k,x}^*)^{-1} J'_k(S_{k,u}(x))$, so can further estimate
    \[
        \begin{split}
             & \norm{S_{k,w}(\thisx) - S_{k+1,w}(\thisx)}
            =
            \norm{(A_{k,x}^*)^{-1} J'_k(S_{k,u}(x)) -(A_{k+1,x}^*)^{-1} J'_k(S_{k+1,u}(x))}
            \\
             &
            =
            \norm{((A_{k,x}^*)^{-1}-(A_{k+1,x}^*)^{-1})J'_k(S_{k,u}(x))-(A_{k+1,x}^*)^{-1   } \left(J'_k(S_{k+1,u}(x)) - J'_k(S_{k,u}(x)) \right)}
            \\
             &
            \le
            \norm{(A_{k,x}^*)^{-1} - (A_{k+1,x}^*)^{-1}}\norm{J'_k(S_{k,u}(x))}
            +\norm{(A_{k+1,x}^*)^{-1}}\norm{J'_k(S_{k+1,u}(x)) - J'_k(S_{k,u}(x))}.
        \end{split}
    \]
    This can, again, be controlled if the dependence of $(A_{k,x}^*)^{-1}$ and $J'_k(S_{k,u}(x))$ on the data (encoded in the index $k$) is continuous.
\end{remark}

\subsection{The differential transformation}
\label{sec:inner-adjoint:differential-transformation}

Finally, we provide conditions for the differential transformation component of \cref{ass:tracking:main}\,\cref{item:tracking:main:differential-transformation} to hold for the reduced adjoint.
As before, we omit subscripts in the norm notation and understand $\norm{.}$ as the norm of the respective underlying space.

\begin{lemma}[Differential transformation: reduced adjoint]
    \label{lemma:tracking:diff-transformation:reduced}
    Continuing from \cref{ex:tracking:pde-adjoint}, take $S_{k,w}(x)=w$ as a solution of \eqref{eq:tracking:reduced-adjoint}.
    Suppose $\diffwrt{T_k}{x}(\freevar, x)$ is $L_{\diffwrt{T}{x}; u}$-Lipshitz for all $x \in \localset_k$ and $k \in \N$; that
    \begin{gather*}
        \sup_{k, \in \N,\, u \in U_k,\, x \in \localset_k} \norm{\diffwrt{T_k}{x}(u, x)}=: M_{\diffwrt{T}{x}} < \infty
        \shortintertext{and that the adjoint solution mappings are bounded on $\localset_k$, i.e.,}
        N_{S_{k,w}} \defeq \sup_{k \in \N,\, x \in \localset_k} \norm{S_{k,w}(x)} < \infty.
    \end{gather*}
    Also let
    \[
        \nextestdiff \defeq \nexxt w \diffwrt{T_{k+1}}{x}(\nexxt u, \predicted{x}{k+1}).
    \]
    Then the differential transformation \cref{ass:tracking:main}\,\cref{item:tracking:main:differential-transformation} holds
    with $\alpha_u =  N_{S_{k,w}} L_{\diffwrt{T}{x}; u}$, $\alpha_w =  M_{\diffwrt{T}{x}}$, and $\transformError{k}=0$.
\end{lemma}

\begin{proof}Write $\nexxt{\hat w} \defeq S_{k+1,w}(\predicted{x}{k+1})$.
    Since
    \[
        E_{k+1}'(\predicted{x}{k+1}) = J_{k+1}'(S_{k+1,u}(\predicted{x}{k+1}) )S_{k+1,u}'(\predicted{x}{k+1}) = \nexxt {\hat w}\diffwrt{T_{k+1}}{x}(S_{k+1,u}(\predicted{x}{k+1}), \predicted{x}{k+1} ),
    \]
    it follows
    \[
        \begin{aligned}[t]
            \norm{\nextestdiff & - E_{k+1}'(\predicted{x}{k+1})}
            =
            \norm{\nexxt w \diffwrt{T_{k+1}}{x}(\nexxt u, \predicted{x}{k+1}) - \nexxt {\hat w}\diffwrt{T_{k+1}}{x}(S_{k+1,u}(\predicted{x}{k+1}), \predicted{x}{k+1} )}%
            \\
             &
            = \norm{[\nexxt w -\nexxt{\hat w}]\diffwrt{T_{k+1}}{x}(\nexxt u, \predicted{x}{k+1}) - \nexxt {\hat w}[\diffwrt{T_{k+1}}{x}(S_{k+1,u}(\predicted{x}{k+1}), \predicted{x}{k+1} )-\diffwrt{T_{k+1}}{x}(\nexxt u, \predicted{x}{k+1}) ]}%
            \\
             &
            \le
            \norm{\diffwrt{T_{k+1}}{x}(\nexxt u, \predicted{x}{k+1})}%
            \norm{\nexxt w -\nexxt{\hat w}}%
            \\
            \MoveEqLeft[-1]
            + \norm{\nexxt {\hat w}}%
            \norm{\diffwrt{T_{k+1}}{x}(S_{k+1,u}(\predicted{x}{k+1}), \predicted{x}{k+1} )-\diffwrt{T_{k+1}}{x}(\nexxt u, \predicted{x}{k+1})}%
        \end{aligned}
    \]
    Therefore, recalling that $S_{k+1,w}(\predicted{x}{k+1})=\nexxt{\hat w}$, we prove \cref{ass:tracking:main}\,\cref{item:tracking:main:differential-transformation}:
    \[
        \norm{\nextestdiff - E_{k+1}'(\predicted{x}{k+1})}%
        \le
        N_{S_{k+1,w}} L_{\diffwrt{T}{x}; u}\norm{\nextu- S_{k+1,u}(\predicted{x}{k+1})}%
        + M_{\diffwrt{T}{x}} \norm{\nexxt w - S_{k+1,w}(\predicted{x}{k+1})}.
        \qedhere
    \]
\end{proof}

\subsection{Linear system splitting}
\label{sec:inner-adjoint:linear}

We now cover specific algorithms for PDEs as the inner problems.
For $U_k, X_k$, and $W_k^*$ normed spaces, let both $A_{k,x} \in \linear(U_k; W_k^*)$, modelling a linear PDE parametrised $x$, and the right hand side $b_{k,x} \in W_k^*$ be Lipschitz in $x \in X_k$.
Consider the inner constraint of $u=S_{k,u}(x)$ satisfying
\begin{equation}
    \label{eq:tracking:linear-inner}
    A_{k,x} u = b_{k,x}.
\end{equation}
This is again an instance of  \eqref{eq:tracking:implicit-inner} when we set
\[
    T_{k}(u, x) = A_{k,x} u - b_{k,x}.
\]
Two solve \eqref{eq:tracking:linear-inner} inexactly and efficiently, we split $A_{k,x}$ into two components, as in standard Gauss–Seidel or Jacobi splittings.
\begin{assumption}[Admissible splitting]%
    \label{ass:primal-admissible-splitting}
    Let $X_k$, $U_k$, and $W_k^*$ be normed spaces.%
    For all $k \ge 0$ and $x^k \in \Omega \subset X_k$,
    let $A_{k,x^k} \in \linear(U_k; W_k^*)$. We assume to be given splittings $ A_{k,\predicted{x}{k}}=N_{k,\predicted{x}{k}}+M_{k,\predicted{x}{k}}$, where $N_{k,\predicted{x}{k}}$ is invertible and
    \[
        \zeta_u \norm{\inv N_{k,\predicted{x}{k}} M_{k,\predicted{x}{k}}} \le 1
    \]
    for some $\zeta_u > 1$.%
\end{assumption}

\begin{assumption}[Adjoint admissible splitting]
    \label{ass:adjoint-admissible-splitting}
    Let
    $X_k$, $U_k$, and $W_k^*$ be normed spaces.%
    For all $k \ge 0$ and $x^k \in \Omega \subset X_k$,
    let $A_{u, k,\predicted{x}{k}} \in \linear(W_k;U_k^*)$. We assume to be given splittings $ A_{u, k,\predicted{x}{k}}=N_{u, k,\predicted{x}{k}}+M_{u, k,\predicted{x}{k}}$, where $N_{k,\predicted{x}{k}}$ is invertible and
    \[
        \zeta_w \norm{\inv N_{u, k,\predicted{x}{k}} M_{u, k,\predicted{x}{k}}} \le 1
    \]
    for some $ \zeta_w > 1$, and that $\norm{N^{-1}_{u, k,\predicted{x}{k}}} \le \gamma$ for some $\gamma \ge 0$.
\end{assumption}

\begin{theorem}\label{thm:tracking:inner-linear-system-splitting}
    The following hold:
    \begin{enumerate}[label=(\roman*)]
        \item\label{item:tracking:inner-linear-system-splitting:inner}
              Suppose \cref{ass:primal-admissible-splitting} holds for $A_{k,\predicted{x}{k}} = N_{k,\predicted{x}{k}} + M_{k,\predicted{x}{k}}$, and that $S_{k,u}(x)=\inv A_{k,x} b_{k,x}$ is $L_{S_u}$-Lipschitz in $\Omega$ for all $k\geq 0$. Assume that $Q_k$ is $L_Q$-Lipschitz for all $k\geq 0$. Then, the inner updates
              \[
                  \nextu = \inv N_{k+1,\predicted{x}{k+1}}(b_{k+1, \predicted{x}{k+1}} - M_{k+1,\predicted{x}{k+1}} \predicted{u}{k+1})
              \]
              satisfy the contractivity condition \eqref{eq:inner-adjoint:inner-contr} with $\bar \kappa_u = \zeta_u $. In particular,
              \cref{ass:tracking:main} \cref{item:tracking:main:inner-tracking} holds if $\zeta_u > L_Q$, i.e. we have,
              \begin{gather*}
                  \begin{split}
                      \inv L_Q\zeta_u \norm{u^{k+1} - S_{k+1,u}(\predicted{x}{k+1})}
                       &
                      \leq
                      \norm{u^k - S_{k,u}(\predicted{x}{k})} + L_{S_u} \norm{\predicted{x}{k} - \thisx}+ \innerError{k}
                  \end{split}
                  \shortintertext{for}
                  \innerError{k} \defeq \inv L_Q \norm{Q_k(S_{k,u}(\thisx)) - S_{k+1,u}(\predicted{x}{k+1})}_{U_{k+1}}.
              \end{gather*}
        \item\label{item:tracking:inner-linear-system-splitting:adjoint}
              Suppose \cref{ass:adjoint-admissible-splitting} holds for $A_{k,x}^* = N_{*, k,\predicted{x}{k}} + M_{*, k,\predicted{x}{k}}$. Assume that, for all $k\geq 0$, $\Tilde Q_k$ is $L_{\tilde Q}$-Lipschitz, $S_{k,w}$ is $L_{S_w}$-Lipschitz, and $J_{k+1}^{'}$ is $L_{J}$-Lipschitz.
              Then, the adjoint updates
              \[
                  w^{k+1} = - N^{-1}_{*, k,\predicted{x}{k+1}} \left(J_{k+1}^{'}(S_{k+1,u}(\predicted{x}{k+1})) + M_{*, k+1,\predicted{w}{k+1}}\predicted{w}{k+1}\right)
              \]
              satisfy the contractivity condition \cref{eq:inner-adjoint:adjoint-contr} with $\bar \kappa_w = \zeta_w $, and $L_w = \zeta_w L_J \gamma$.
            In particular, \cref{ass:tracking:main} \cref{item:tracking:main:adjoint-tracking} holds if $\zeta_w > L_{\tilde Q}$, i.e. we have,%
              \begin{gather*}
                  \begin{split}
                      \frac{\zeta_w}{L_{\tilde Q}} \norm{w^{k+1} - S_{k+1,w}(\predicted{x}{k+1})}_{W_{k+1}}
                       &
                      \le
                      \norm{w^k - S_{k,w}(\predicted{x}{k})}_{W_k}
                      + \frac{\zeta_w L_J \gamma}{L_{\tilde Q}} \norm{u^{k+1} - S_{k+1,u}(\predicted{x}{k+1})}_{U_{k+1}}
                      \\
                      \MoveEqLeft[-1]
                      + L_{S_w}\norm{\predicted{x}{k} - \thisx}_{X_k}  + \adjointError{k}
                  \end{split}
                  \shortintertext{for}
                  \adjointError{k} = \inv L_{\tilde Q} \norm{\Tilde{Q}_k(S_{k,w}(\thisx)) - S_{k+1,w}(\predicted{x}{k+1})}_{W_{k+1}}.
              \end{gather*}
        \item\label{item:tracking:inner-linear-system-splitting:dt}
              Assume that $u \mapsto  A_{k+1,\predicted{x}{k+1}}^{(x)}u - b_{k+1,\predicted{x}{k+1}}^{(x)}$ is $L_{T,u}$-Lipschitz, and $\norm{A_{k+1,\predicted{x}{k+1}}^{(x)}u^{k+1} - b_{k+1,\predicted{x}{k+1}}^{(x)}} \leq M_{T^{(x)}}$. Then, \cref{ass:tracking:main} \cref{item:tracking:main:differential-transformation} holds:
              \[
                  \norm{\nextestdiff - E_{k+1}'(\predicted{x}{k+1})}%
                  \le
                  M_{T^{(x)}}\norm{\nexxt w - S_{k+1,w}(\predicted{x}{k+1})} +
                  L_{T,u} \norm{\nexxt {\hat w}} \norm{\nexxt u - S_{k+1,u}(\predicted{x}{k+1})}.
              \]
    \end{enumerate}
\end{theorem}\begin{proof}\cref{item:tracking:inner-linear-system-splitting:inner}:
    We have
    \begin{equation}
        \label{eq:tracking:linear-system-splitting:one-step}
        \begin{split}
            \nextu - S_{k+1,u}(\predicted{x}{k+1})
             & =\inv N_{k+1,\predicted{x}{k+1}}(b_{k+1, \predicted{x}{k+1}} - M_{k+1,\predicted{x}{k+1}} \predicted{u}{k+1})
            \\
             & \qquad - \inv N_{k+1,\predicted{x}{k+1}}(b_{k+1, \predicted{x}{k+1}} - M_{k+1,\predicted{x}{k+1}}S_{k+1,u}(\predicted{x}{k+1}))
            \\
             & =-\inv N_{k+1,\predicted{x}{k+1}}M_{k+1,\predicted{x}{k+1}}(\predicted{u}{k+1} - S_{k+1,u}(\predicted{x}{k+1})).
        \end{split}
    \end{equation}
    Hence,
    \[
        \zeta_u \norm{\nextu - S_{k+1,u}(\predicted{x}{k+1})}
        \le
        \zeta_u \norm{\inv N_{k+1,\predicted{x}{k+1}}M_{k+1,\predicted{x}{k+1}}}\norm{\predicted{u}{k+1} - S_{k+1,u}(\predicted{x}{k+1})}
        \le
        \norm{\predicted{u}{k+1} - S_{k+1,u}(\predicted{x}{k+1})}.
    \]
    Thus, using \cref{lemma:inner-adjoint:inner-simple-contractive},  \cref{ass:tracking:main}\,\cref{item:tracking:main:inner-tracking} holds.

    \cref{item:tracking:inner-linear-system-splitting:adjoint}:
    For brevity, set $v_k = (*,k,\predicted{x}{k})$. Hence, using the definition of $w^{k+1}$, and the fact that
    \[
        (M_{v_{k+1}} + N_{v_{k+1}}) S_{k+1,w}(\predicted{x}{k+1}) = -J_{k+1}' (S_{k+1,u}(\predicted{x}{k+1})),
    \]
    we obtain
    \begin{equation}
        \label{eq:tracking:linear-system-splitting:adjoint-one-step}
        \begin{split}
        w^{k+1} & - S_{k+1,w}(\predicted{x}{k+1}) = -N^{-1}_{v_{k+1}}\left(J_{k+1}'(S_{k+1,u}(\predicted{x}{k+1})) + M_{v_{k+1}}\predicted{w}{k+1}\right) - S_{k+1,w}(\predicted{x}{k+1})
        \\
                                                & = -N^{-1}_{v_{k+1}}\left(J_{k+1}' (S_{k+1,u}(\predicted{x}{k+1}))  -  J'_{k+1}(u^{k+1})\right) -N^{-1}_{v_{k+1}}M_{v_{k+1}}\predicted{w}{k+1}
        \\
                                                & \qquad  - N^{-1}_{v_{k+1}}J'_{k+1}(u^{k+1}) - S_{k+1,w}(\predicted{x}{k+1})
        \\
                                                & = -N^{-1}_{v_{k+1}}\left(J_{k+1}' (S_{k+1,u}(\predicted{x}{k+1}))  -  J'_{k+1}(u^{k+1})\right) -N^{-1}_{v_{k+1}}M_{v_{k+1}}\predicted{w}{k+1}
        \\
                                                & \qquad  + N^{-1}_{v_{k+1}}\left((M_{v_{k+1}} + N_{v_{k+1}})S_{k+1,w}(\predicted{x}{k+1})\right)- S_{k+1,w}(\predicted{x}{k+1})
        \\
                                                & = -N^{-1}_{v_{k+1}}\left(J_{k+1}' (S_{k+1,u}(\predicted{x}{k+1}))  -  J'_{k+1}(u^{k+1})\right) + N^{-1}_{v_{k+1}}M_{v_{k+1}}\left( S_{k+1,u}(\predicted{x}{k+1}) - \predicted{w}{k+1}\right).
        \end{split}
    \end{equation}
    Thus, using the Lipschitz properties of $J_{k+1}'$ and \cref{ass:adjoint-admissible-splitting}, we obtain
    \begin{align*}
        \norm{w^{k+1} - S_{k+1,w}(\predicted{x}{k+1})}_{W_{k+1}}
         & \leq
        \gamma L_J \norm{u^{k+1} - S_{k+1,u}(\predicted{x}{k+1})}_{U_{k+1}} + \frac{1}{\zeta_w} \norm{\predicted{w}{k+1} - S_{k+1,w}(\predicted{x}{k+1})}_{W_{k+1}}.
    \end{align*}
    This proves our contractivity condition \cref{eq:inner-adjoint:adjoint-contr} with $\bar \kappa_w = \zeta_w$ and $L_w = \zeta_w L_J \gamma $ .
    Thus, using \cref{lemma:inner-adjoint:adjoint-simple-contractive}, \cref{ass:tracking:main}\,\cref{item:tracking:main:adjoint-tracking} holds.

    \cref{item:tracking:inner-linear-system-splitting:dt}: Finally, we have
    \[
        \begin{aligned}[t]
             & \norm{\nextestdiff - E_{k+1}'(\predicted{x}{k+1})}
            =
            \norm{\nexxt w \diffwrt{T_{k+1}}{x}(\nexxt u, \predicted{x}{k+1}) -  S_{k+1,w}(\predicted{x}{k+1})\diffwrt{T_{k+1}}{x}(S_{k+1,u}(\predicted{x}{k+1}), \predicted{x}{k+1} )}%
            \\
             &
            =
            \norm{[\nexxt w - S_{k+1,w}(\predicted{x}{k+1})]\diffwrt{T_{k+1}}{x}(\nexxt u, \predicted{x}{k+1}) - \nexxt {\hat w}[\diffwrt{T_{k+1}}{x}(S_{k+1,u}(\predicted{x}{k+1}), \predicted{x}{k+1} )-\diffwrt{T_{k+1}}{x}(\nexxt u, \predicted{x}{k+1}) ]}%
            \\
             &
            \le
            \norm{\diffwrt{T_{k+1}}{x}(\nexxt u, \predicted{x}{k+1})}%
            \norm{\nexxt w - S_{k+1,w}(\predicted{x}{k+1})}
            \\
            \MoveEqLeft[-1]
            + \norm{\nexxt {\hat w}}
            \norm{\diffwrt{T_{k+1}}{x}(S_{k+1,u}(\predicted{x}{k+1}), \predicted{x}{k+1} )-\diffwrt{T_{k+1}}{x}(\nexxt u, \predicted{x}{k+1})}
            \\
             &
            \le
            M_{T^{(x)}}\norm{\nexxt w - S_{k+1,w}(\predicted{x}{k+1})} + L_{T,u} \norm{\nexxt {\hat w}} \norm{\nexxt u - S_{k+1,u}(\predicted{x}{k+1})}.
        \end{aligned}
    \]
    Thus,  \cref{ass:tracking:main}\,\cref{item:tracking:main:differential-transformation} holds.
\end{proof}

\begin{example}[Standard linear system splitting schemes]
    \label{ex:splitting:Gauss–Seidel}
    For Jacobi splitting in $U=\R^m$, $N_{k,x}$ is the diagonal of $A_{k,x}$.
    For Gauss–Seidel splitting, $N_{x,k}$ is the lower triangle.
    In general, \cref{ass:primal-admissible-splitting,ass:adjoint-admissible-splitting} have to be verified on a case-by-case basis.
\end{example}

\begin{remark}[Multiple steps and spectral radii]
    \label{rem:tracking:spectral}
    Let $U \sim \R^m$ be equipped with the $2$-norm, so that the norm in $\linear(U; U)$ is the spectral norm.
    Perform $n_{k+1} \ge 1$ inner steps by setting $u^{k+1} \defeq u^{k+1}_{n_{k+1}}$, $u^{k+1}_0 \defeq u^k$, and $u^{k+1}_{j+1} \defeq \inv N_{k+1,\predicted{x}{k+1}} (b_{k+1,\predicted{x}{k+1}} - M_{k+1,\predicted{x}{k+1}}u^{k+1}_j)$ over $j=0,\ldots,n_{k+1}-1$.
    By recursive application of \cref{eq:tracking:linear-system-splitting:one-step}, \cref{thm:tracking:inner-linear-system-splitting}\,\cref{item:tracking:inner-linear-system-splitting:inner} extends to these updates, if we modify \cref{ass:primal-admissible-splitting} to require $\zeta_u \norm{(\inv N_{k,\predicted{x}{k}} M_{k,\predicted{x}{k}})^{n_k}}_{\linear(U; U)} \le 1$.
    Similarly to the proof of \cite[Theorem 10.1.1]{golub1996matrix}, this norm bound can be ensured for any given $\zeta_u > 1$ for some $n_k$ through \cite[Lemma 7.3.2]{golub1996matrix} if the spectral radius $\rho(\inv N_{k,\predicted{x}{k}} M_{k,\predicted{x}{k}}) < 1$.
    The latter holds, for example, for the Gauss--Seidel splitting when $A_x$ is strictly diagonally dominant \cite[§10.1]{golub1996matrix}.

    Likewise, adapting \cref{ass:adjoint-admissible-splitting} to require $ \zeta_w \norm{(\inv N_{u, k,\predicted{x}{k}} M_{u, k,\predicted{x}{k}})^{n_k}} \le 1$, and applying \eqref{eq:tracking:linear-system-splitting:adjoint-one-step} recursively, \cref{thm:tracking:inner-linear-system-splitting}\,\cref{item:tracking:inner-linear-system-splitting:adjoint} extends to $n_k$ adjoint steps.

    In general, if \cref{ass:primal-admissible-splitting} holds for some $\tilde \zeta_u>1$, then repeating the inner step sufficiently many steps, we can make \cref{ass:primal-admissible-splitting} hold for an arbitrary $\zeta_u>1$ for the multi-step inner method.
\end{remark}

\begin{remark}[Full error expression]
    \label{rem:inner-adjoint:error}
    Subject to \cref{ass:tracking:main}, verified in \cref{thm:tracking:inner-linear-system-splitting}, and $E$ itself satisfying smoothness properties that correspond to \cref{ass:tracking:main} (see \cref{cor:3pt-to-ass-smoothness} for the quadratic case), \cref{thm:tracking:smoothness-verification} verifies \cref{ass:pd:smoothness} for the differential estimates.
    Therefore, to apply the regret \cref{cor:pd:main} to \cref{alg:pd:alg}, we merely need the overall problem structure of \cref{ass:pd:main}, the critical point proximity conditions of \cref{ass:pd:add-local}, and to bound the errors
    \[
        e_N^\Sigma(u^{0:N-1},\optu^{0:N})
        =
        \sum_{k=0}^{N-1}\!\left(
        \epsilon_{k+1}^{\dagger}(u^k,\optu^{k:k+1})
        + \tau\Err_{k+1}
        \right).
    \]
    In the present setting, the prediction error $\epsilon_{k+1}^{\dagger}(u^k, \bar u^{k:{k+1}})$, defined in \cref{eq:pd:prediction-error}, depends on the predictors use, and is discussed in more detail in \cite{tuomov-better-predict,tuomov2024online-eit}.
    From \cref{thm:tracking:smoothness-verification}\,\cref{item:tracking:smoothness-verification:local}, the algorithm error satisfies $\Err_{k+1} = \frac{1}{2\tilde\gamma}\combinederror_k$, where $\combinederror_k$ is defined in \cref{lemma:tracking:smoothness-pre-result}. Moreover, $\combinederror_k$ can be bounded as
    \begin{align*}
         \frac{1}{2\tilde\gamma}\sum_{k=0}^{N}\combinederror_k
         &
        \defeq  \frac{1}{2\tilde\gamma}\sum_{k= 0}^{N} \left(\frac{5}{2}\trackingressum[1]^2 \left(\max\!\left( \frac{1}{\pi_u}\,\innerError{k}, \frac{1}{\pi_w}\,\adjointError{k} \right)\right)^2 + e_{1,k}\right)
        \\
         & \leq  \frac{1}{2\tilde\gamma}\sum_{k= 0}^{N} \left(\frac{5}{2}\trackingressum[1]^2 \left(\max\!\left( \frac{1}{\pi_u}\,\innerError{k}, \frac{1}{\pi_w}\,\adjointError{k} \right)\right)^2\right) +  \frac{1}{2\tilde\gamma} \Psi_1,
    \end{align*}
    where $\Psi_1$ and $\trackingressum[1]$ are bounded in \cref{eq:tracking:smoothness-pre-result:e-bound,eq:tracking:smoothness-pre-result:varrho-bound} based on the first-step errors and the coefficients $\alpha_u = L_{T,u}$, $\alpha_w = M_{T^{(x)}}$, $\pi_u = L_{S_u}$, and $\pi_w = L_{S_w}$.

    When $Q_k=\Id$ and $\tilde Q_k=\Id$, as in our numerical experiments of \cref{sec:numerical}, we obtain from \cref{rem:iner-adjoint:sol-map-data-error,rem:tracking:sol-map-data-error-adjoint} that
    \begin{align*}
        \innerError{k}
         &
        \leq
        \norm{(\inv A_{k,\thisx}-\inv A_{k+1,\thisx})}\norm{b_{k,\thisx}}
        +\norm{\inv A_{k+1,\thisx}}\norm{b_{k+1,\thisx}-b_{k,\thisx}}
        +L_{S,u}\norm{(P_k-\Id)(\thisx)},
        \\
        \adjointError{k}
         &
        \leq
        \norm{(A_{k,x}^*)^{-1} - (A_{k+1,x}^*)^{-1}}\norm{J'_k(\inv A_{k,x} b_{k,x})}
        \\
        \MoveEqLeft[-1]
        +\norm{(A_{k+1,x}^*)^{-1}}\norm{J'_k(\inv A_{k+1,x} b_{k+1,x}) - J'_k(\inv A_{k,x} b_{k,x})}
        +L_{S,w}\norm{(P_k-\Id)(\thisx)}.
    \end{align*}
    These errors can be further bounded through the specific appearance of the frame-dependent data in $A_{k,x}$, $b_{k,x}$, and $J_k$, as well as the difference of the primal predictor $P_k$ from the identity.
\end{remark}

\section{Numerical experiments}
\label{sec:numerical}

In this section, in the context of application of the online \cref{alg:pd:alg} to the dynamical electrical impedance tomography (EIT) problem
\crefrange{eq:eit:cem}{eq:online:problem},

In the dynamical EIT problem, the goal is to reconstruct a time-varying conductivity distribution containing moving inclusions. The evolution of the inclusions is governed by a transport equation corresponding to incompressible flow with constant density. In addition to this nominal setting, we deliberately consider scenarios that violate the modelling assumptions in order to assess the robustness of the method.

The numerical experiments, including the EIT forward model, data generation, discretisation choices, and noise model, closely follow the setup introduced in \cite{tuomov2024online-eit}. We therefore only summarise the key ingredients here and refer the reader to that work for full details.
Our computations were performed on a 2020 MacBook Air M1 with 16GB memory.
Our Julia implementation of the experiments and algorithms is available on Zenodo \cite{jauhiainen2025online-eit-codes}.

\subsection{Test scenarios}
\label{sec:numerical:scenarios}

All experiments are performed on synthetic data in a disk-shaped domain $\Omega$ equipped with $\Nelec = 16$ equally spaced boundary electrodes.
The data fidelity term $E_k$ defined in \eqref{eq:eit:functionals} is evaluated using a finite element approximation of the electrode potentials and the conductivity, both represented with piecewise linear basis functions.
To avoid the inverse crime \cite{kaipio2000statistical}, the mesh used for the inverse problem is coarser (2917 nodes) than that used for data generation (5039 nodes). All remaining modelling and discretisation choices follow \cite{tuomov2024online-eit}.

For each time instance, current injection is applied sequentially at a single electrode while all remaining electrodes are grounded, resulting in $\Nmeas = 16$ excitation patterns. As is standard in EIT, the current measured at the excited electrode is excluded from the data, yielding $(\Nelec - 1)\Nmeas = 240$ measurements per frame.

We consider the following four scenarios:
\begin{description}\item[Constant Motion]
          In this baseline experiment, a single inclusion moves with constant velocity in a homogeneous background. This serves as a reference case.

    \item[Circular Motion]
          A single inclusion follows a circular trajectory, violating the constant-velocity assumption.

    \item[Halting Motion]
          A single inclusion stops completely at frames 1000 and 2000, further challenging the prediction model.

    \item[Disappearing Inclusions]
          Two inclusions move along circular paths. One inclusion disappears at frame 500 and the other at frame 1000; both reappear at frame 1500. This scenario violates the incompressibility assumption.
\end{description}

The \emph{Constant Motion} experiment consists of 400 frames, while the remaining experiments contain 2000 frames. The background conductivity is fixed at $x_{\mathrm{bg}} = 1\,\mathrm{S}$, and all inclusions are resistive with conductivity $x_{\mathrm{incl}} = 10^{-4}\,\mathrm{S}$. Measurement data are simulated using a finite element approximation of the complete electrode model \eqref{eq:eit:cem}, with additive Gaussian noise at a relative level of $10^{-4}$. Further details on the forward solver and derivative computations can be found in \cite{jauhiainen2020ripgn}.

Unless stated otherwise, all experiments use identical algorithmic parameters. These are chosen in accordance with the guidelines of \cite{tuomov2024online-eit}.%

\subsection{Algorithms and their parameters}
\label{sec:numerical:algorithms}

We compare the numerical performance of the following instances of \cref{alg:pd:alg}.
The difference between the instances is how $\nextestgrad$ is formed.
All methods are based on the reduced adjoint equation of \cref{ex:tracking:pde-adjoint} (or, more precisely, a version with a change-of-basis as described in \cite{tuomov2024tracking}), but differ in how the inner iterates $\nextu$ and adjoint iterates $\nexxt w$ are calculated.

\begin{description}\item[Gauss–Seidel 7 + 1]
          This is our proposed single-loop primal-dual method, taking on each iteration of the outer optimisation method, 7 Gauss–Seidel iterations towards the solution the EIT forward PDE to form $\nextu$ from $\thisu$, and 1 Gauss–Seidel iteration towards the solution of the reduced adjoint PDE to form $\nexxt w$ from $\this w$.
          This follows \cref{thm:tracking:inner-linear-system-splitting,ex:splitting:Gauss–Seidel}.
          The errors for \cref{cor:pd:main} are expanded in \cref{rem:inner-adjoint:error}.

    \item[Exact]
          In this variant of our method, we solve the PDE and its reduced adjoint exactly,
          taking (up to numerical precision) $\nextestgrad=\nextgrad$, $\nextu=S_u(\thisx)$, and $\nexxt w=S_w(\thisx)$.
          The parametrisation of the method is otherwise exactly the same as of the previous one.

    \item[Background]
          This is the method of \cite{tuomov2024online-eit} with the same step length parameters as we used with the above two methods.
          This method solves the PDE and the adjoint PDE in a background computational thread.
          While waiting for the solution to become available, the main thread proceeds with the iterations of the optimisation method with new data frames, using the modified data term
          \[
              \check E(x) \defeq \frac{1}{2}\norm{I(\check x) + \grad I(\check x)^*(x-\check x)-b}_{\inv\Sigma}^2
          \]
          that linearises the solution operator $I$ at a past iterate (linearisation point) $\check x$.
          This model necessitates the full basic adjoint, instead of the reduced adjoint.
          Once the new solution becomes available, it updates the linearisation point, and asks the background solver to solve the PDE at its own current iterate.
    \item[Background, original step lengths]
          This is the method used in the numerical experiments of \cite{tuomov2024online-eit}, in which we take $\tau = 0.85/\check L^2$ and $\sigma = 1$, where $\check x$ is the current linearisation point of the background solver, and $\check L \defeq \norm{\WOp\grad I(\check x)}$.
\end{description}

\begin{remark}[On the number of inner and adjoint iterations]
    The number of iterations required in the first single-loop approach is connected, in particular, to sufficient contractivity factors $\kappa_u$ and $\kappa_w$ and sufficiently small errors in \cref{ass:tracking:main} for \cref{ass:pd:add-local} to hold for the chosen step length and other parameters; see \cref{ex:splitting:Gauss–Seidel,rem:tracking:spectral}.
    The connection between the assumptions is made through the smoothness errors that were expanded in \cref{rem:inner-adjoint:error}.
    The higher the number of iterations, the higher the contractivity.
    The base contractivity, for a single iteration, depends on the PDE constraint in question, and the range of the data in it: in case of the present problem
    \crefrange{eq:eit:cem}{eq:online:problem},
    the conductivity $x^k$, which is explicitly bounded, and the chosen electrode potentials $U_i^{j,k}$.
    The configuration “7 + 1” has been found by trial and error to work well in all experiments.
\end{remark}

The step length parameters for the three first algorithms are fixed $\tau=0.005$. $\sigma = 6.0$. The scheme of \cite{tuomov2024online-eit} cannot be justified anymore, as we no longer linearise the solution operator. For $\check E$, $\check L$ above gives a valid step length for the convex PDPS. Without the linearisation, we could, in a heuristic fashion, adaptively updated the step lengths based on $L_k \defeq \max_{j \le k} \norm{\grad^2 \tilde E(x^j)}$, however, due to the lack of linearisation, this would significantly increase the computational by demanding $\grad^2 I(x^j)$.

Also, although we use the method and parametrisation of \cite{tuomov2024online-eit} to benchmark our own methods, the performance we obtain from the method is not exactly the same as in \cite{tuomov2024online-eit}, as we significantly optimised the code.
Due to its interleaved, threaded structure, the method is very sensitive to underlying hardware, other operating system events, and code performance.
Although we have also optimised the PDE solver, especially the main computational thread is much faster than it was in \cite{tuomov2024online-eit}, which causes much more frames to be processed while waiting for the PDE solution.
This results in additional instability that we observe, in particular, in the \emph{Halting Motion} and \emph{Disappearing Inclusions} experiments. Slowing down the method---further limiting how far in history the linearisation point can be---could reduce the instability at additional computational cost.

\subsection{Results}
\label{sec:numerical:results}

\Cref{fig:PDP1:value,fig:PDP2:value,fig:PDP1:gterror,fig:PDP2:gterror,tab:summary} show the relative objective values and iterate errors
\begin{align}
    \label{eq:relmeasures}
    J_{\operatorname{rel}}^k \defeq \frac{J_k(x^k)}{J_k(x^0)}
    \quad\text{and}\quad
    e_{\operatorname{rel}}^k \defeq \frac{\norm{x^k-x_{\text{true}}^k}}{\norm{x_{\text{true}}^k}},
\end{align}
where $x_{\text{true}}^k$ denotes the ground truth at iteration~$k$, for the tested predictor configurations.

We illustrate the behaviour of the algorithm in terms of function values \cref{fig:PDP1:value,fig:PDP2:value} for no prediction and for the best-performing outer primal-dual predictor $P_k$ from \cite{tuomov2024online-eit,tuomov-better-predict}: optical flow for the primal variable $x$, and a scaling dual predictor for $y$.
For the inner and adjoint predictor of the variables $u$ and $w$ we set, for simplicity, $Q_k=\Id$ and $\tilde Q_k=\Id$.
Our numerical results could possibly be improved slightly by developing more fine-tuned predictors for these variables.
Then the error terms of the regret \cref{cor:pd:main} are given by \cref{rem:inner-adjoint:error}.
The corresponding plots of relative error to the ground truth are in \cref{fig:PDP1:gterror,fig:PDP2:gterror}.

\begingroup

\def\NPName{\emph{No Prediction}}
\expandafter\def\csname PDP1Name\endcsname{the \emph{Greedy} predictor}
\expandafter\def\csname PDP2Name\endcsname{the \emph{Affine} predictor}

\def\doplotGTError#1{%
    \def\motion{#1}%
    \tikzsetnextfilename{\mesh_\motion\alphaparam_\pred_gt_rel_error}
    \begin{tikzpicture}
        \begin{axis}[
                width=0.45\textwidth,
                height=0.2\textwidth,
                grid=major,
                title={\motion},
                scaled x ticks=false,
                xminorticks=true,
                yminorticks=true,
                axis x line*=bottom,
                axis y line*=left,
                legend style={
                        legend pos=north east,
                        /tikz/column sep=1ex,
                        draw=none,
                        fill=none,
                        legend columns=-1,
                    },
                y tick label style={
                        /pgf/number format/fixed,
                        /pgf/number format/precision=2
                    },
                legend to name = {leg:\pred:\motion:gterror},
                scale only axis,
                x tick label style={
                        /pgf/number format/1000 sep={\ },
                    },
            ]
            \addplot [color=Set2-B, dashed, line width=1pt]
            table [x=iter, y=gt_rel_error, col sep=comma]
                {Results/\mesh/\motion\alphaparam/background\origextra_\pred_\bgorigsteplengths/timings.txt};
            \addlegendentry{Background; orig step lengths}

            \addplot [color=Set2-C, dotted, line width=1pt]
            table [x=iter, y=gt_rel_error, col sep=comma]
                {Results/\mesh/\motion\alphaparam/background\bgextra_\pred_\steplengths/timings.txt};
            \addlegendentry{Background}

            \addplot [color=Set2-D, line width=1pt]
            table [x=iter, y=gt_rel_error, col sep=comma]
                {Results/\mesh/\motion\alphaparam/exact\extra_\pred_\steplengths/timings.txt};
            \addlegendentry{Exact}

            \addplot [color=Set2-A, line width=1pt]
            table [x=iter, y=gt_rel_error, col sep=comma]
                {Results/\mesh/\motion\alphaparam/GaussSeidel\gsn\gsextra_\pred_\steplengths/timings.txt};
            \addlegendentry{Gauss--Seidel / \gsnname}
        \end{axis}
    \end{tikzpicture}
}

\def\doplotValue#1{%
    \def\motion{#1}%
    \tikzsetnextfilename{\mesh_\motion\alphaparam_\pred_rel_value}
    \begin{tikzpicture}
        \begin{axis}[
                width=0.45\textwidth,
                height=0.2\textwidth,
                grid=major,
                ymode=log,
                title={\motion},
                scaled x ticks=false,
                xminorticks=true,
                yminorticks=true,
                axis x line*=bottom,
                axis y line*=left,
                legend style={
                        legend pos=north east,
                        /tikz/column sep=1ex,
                        draw=none,
                        fill=none,
                        legend columns=-1,
                    },
                legend to name = {leg:\pred:\motion:value},
                legend columns=-1,
                scale only axis,
                x tick label style={
                        /pgf/number format/1000 sep={\ },
                    },
            ]
            \addplot [color=Set2-B, dashed, line width=1pt]
            table [x=iter, y=rel_value, col sep=comma]
                {Results/\mesh/\motion\alphaparam/background\origextra_\pred_\bgorigsteplengths/timings.txt};
            \addlegendentry{Background; orig. step lengths}

            \addplot [color=Set2-C, dotted, line width=1pt]
            table [x=iter, y=rel_value, col sep=comma]
                {Results/\mesh/\motion\alphaparam/background\bgextra_\pred_\steplengths/timings.txt};
            \addlegendentry{Background}

            \addplot [color=Set2-D, line width=1pt]
            table [x=iter, y=rel_value, col sep=comma]
                {Results/\mesh/\motion\alphaparam/exact\extra_\pred_\steplengths/timings.txt};
            \addlegendentry{Exact}

            \addplot [color=Set2-A, line width=1pt]
            table [x=iter, y=rel_value, col sep=comma]
                {Results/\mesh/\motion\alphaparam/GaussSeidel\gsn\gsextra_\pred_\steplengths/timings.txt};
            \addlegendentry{Gauss--Seidel / \gsnname}
        \end{axis}
    \end{tikzpicture}
}

\def\plotallGTError#1{
    \gdef\pred{#1}
    \def\predname{\csname #1Name\endcsname}

    \begin{figure}
        \centering%
        \pgfplotslegendfromname{leg:\pred:ConstantMotion:gterror}%
        \tikzexternalenable%
        \\
        \setlength{\tabcolsep}{0pt}
        \noindent
        \begin{tabular}{rr}
            \doplotGTError{ConstantMotion}%
             &
            \doplotGTError{HaltingMotion}%
            \\
            \doplotGTError{TurningMotion}%
             &
            \doplotGTError{DisappearingMotion}%
        \end{tabular}
        \tikzexternaldisable%
        \caption{Ground truth relative error for \predname. The horizontal axis indicates the iteration, and the vertical axis the error the relative error to the ground truth \eqref{eq:relmeasures}.}
        \label{fig:\pred:gterror}
    \end{figure}

}

\def\plotallValue#1{
    \gdef\pred{#1}
    \def\predname{\csname #1Name\endcsname}

    \begin{figure}
        \centering%
        \pgfplotslegendfromname{leg:\pred:ConstantMotion:value}%
        \tikzexternalenable%
        \\
        \setlength{\tabcolsep}{0pt}
        \noindent
        \begin{tabular}{rr}
            \doplotValue{ConstantMotion}%
             &
            \doplotValue{HaltingMotion}%
            \\
            \doplotValue{TurningMotion}%
             &
            \doplotValue{DisappearingMotion}%
        \end{tabular}
        \tikzexternaldisable%
        \caption{Relative function value for \predname. The horizontal axis indicates the iteration, and the vertical axis the relative function value \eqref{eq:relmeasures}.}
        \label{fig:\pred:value}
    \end{figure}

}

\def\mesh{Disk16}
\def\alphaparam{0.1}

\def\origextra{_oldstep6}
\def\bgorigsteplengths{sigma1.0_tau0.85}

\def\gsn{7}
\def\gsnname{7+1}
\def\steplengths{sigma10.0_tau0.0053}
\def\bgextra{_fixedstep}
\def\extra{_fixedstep}
\def\gsextra{_ADJGaussSeidel1_fixedstep_upa10}

\plotallValue{PDP1}\plotallGTError{PDP1}
\plotallValue{PDP2}\plotallGTError{PDP2}

\def\illustration#1#2{
    \def\motion{#1}
    \def\pred{#2}
    \def\predname{\csname #2Name\endcsname}

    \def\choseniters{
        \raisebox{-.5\height}{\includegraphics[width=0.12\linewidth]{\pfx_1.png}}
        &
        \raisebox{-.5\height}{\includegraphics[width=0.12\linewidth]{\pfx_400.png}}
        &
        \raisebox{-.5\height}{\includegraphics[width=0.12\linewidth]{\pfx_800.png}}
        &
        \raisebox{-.5\height}{\includegraphics[width=0.12\linewidth]{\pfx_1200.png}}
        &
        \raisebox{-.5\height}{\includegraphics[width=0.12\linewidth]{\pfx_1600.png}}
        &
        \raisebox{-.5\height}{\includegraphics[width=0.12\linewidth]{\pfx_2000.png}}
    }

    \begin{figure}
        \centering
        \setlength{\tabcolsep}{0pt}
        \begin{tabular}{lcccccc}
            Alg.~/ Frame & 1 & 400 & 800 & 1200 & 1600 & 2000
            \\
            \hline
            True
                         &
            \gdef\pfx{Results/\mesh/\motion\alphaparam/true/conductivity}%
            \choseniters%
            \\
            \hline
            Bg.~orig.
                         &
            \gdef\pfx{Results/\mesh/\motion\alphaparam/background\origextra_\pred_\bgorigsteplengths/reco_conductivity}%
            \choseniters%
            \\
            Background
                         &
            \gdef\pfx{Results/\mesh/\motion\alphaparam/background\bgextra_\pred_\steplengths/reco_conductivity}%
            \choseniters%
            \\
            Exact
                         &
            \gdef\pfx{Results/\mesh/\motion\alphaparam/exact\extra_\pred_\steplengths/reco_conductivity}%
            \choseniters%
            \\
            Gauss–Seidel \gsnname
                         &
            \gdef\pfx{Results/\mesh/\motion\alphaparam/GaussSeidel\gsn\gsextra_\pred_\steplengths/reco_conductivity}%
            \choseniters%
        \end{tabular}
        \caption{Reconstructions for the \motion~experiment with \predname.}
    \end{figure}

}

\illustration{HaltingMotion}{PDP2}
\illustration{TurningMotion}{PDP2}
\illustration{DisappearingMotion}{PDP2}

%
%
\tikzifexternalizing{\def\fubar{no}}{\let\fubar\empty}
\ifx\fubar\empty
    \pgfplotstableset{
        col sep=comma,
        every column/.style={
                fixed,
                fixed zerofill,
                precision=2,
            },
        columns/experiment_name/.style = {
                column name = {Exp.},
                string type,
                string replace={ConstantMotion0.1}{Const.},
                string replace={TurningMotion0.1}{Turn.},
                string replace={HaltingMotion0.1}{Halt.},
                string replace={DisappearingMotion0.1}{Dis.},
            },
        columns/predictor_name/.style = {
                column name = {Pred.},
                string type,
                string replace={NP}{No},
                string replace={PDP1}{Gr},
                string replace={PDP2}{Af},
            },
        columns/algorithm_name/.style = {
                column name = {Alg.},
                string type,
                string replace={GaussSeidel7_ADJGaussSeidel1_fixedstep_upa10}{GS},
                string replace={exact_fixedstep}{Exact},
                string replace={background_fixedstep}{Bg.},
                string replace={background_oldstep6}{Orig.},
            },
        columns/re_mean1/.style = {column name = {Mean}},
        columns/re_meanN/.style = {column name = {Mean}},
        columns/re_stdN/.style = {column name = {Std}},
        columns/re_ci/.style = {
                string type,
                column name = {CI}
            },
        columns/re_ci_lower/.style = {column name = {CI lower}},
        columns/re_ci_upper/.style = {column name = {CI upper}},
        columns/re_gt_mean1/.style = {column name = {GT mean}},
        columns/re_gt_meanN/.style = {
                column name = {Mean},
                precision=3,
            },
        columns/re_gt_stdN/.style = {
                column name = {Std},
                precision=3,
            },
        columns/re_gt_ci/.style = {
                string type,
                column name = {CI},
            },
        columns/re_gt_ci_lower/.style = {column name = {GT CI lower}},
        columns/re_gt_ci_upper/.style = {column name = {GT CI upper}},
        columns/total_time/.style = {column name = {Total time}},
        columns/frame_time/.style = {
                column name = {Real},
                precision=4,
            },
        columns/total_cpu_time/.style = {column name = {Total CPU time}},
        columns/frame_cpu_time/.style = {
                column name = {CPU},
                precision=4,
            },
    }

    \pgfplotstableread{Results/Disk16/ConstantMotion0.1/summarystats.csv}{\RES}
    \pgfplotstablevertcat{\RES}{Results/Disk16/TurningMotion0.1/summarystats.csv}
    \pgfplotstablevertcat{\RES}{Results/Disk16/HaltingMotion0.1/summarystats.csv}
    \pgfplotstablevertcat{\RES}{Results/Disk16/DisappearingMotion0.1/summarystats.csv}

    \pgfplotstablecreatecol[
        create col/assign/.code={%
                \getthisrow{re_ci_lower}\low
                \getthisrow{re_ci_upper}\high
                \edef\entry{$[\pgfmathprintnumber[precision=2,fixed]{\low},
                            \pgfmathprintnumber[precision=2,fixed]{\high}]$}%
                \pgfkeyslet{/pgfplots/table/create col/next content}\entry
            }
    ]{re_ci}{\RES}
    \pgfplotstablecreatecol[
        create col/assign/.code={%
                \getthisrow{re_gt_ci_lower}\low
                \getthisrow{re_gt_ci_upper}\high
                \edef\entry{$[\pgfmathprintnumber[precision=3,fixed]{\low},
                            \pgfmathprintnumber[precision=3,fixed]{\high}]$}%
                \pgfkeyslet{/pgfplots/table/create col/next content}\entry
            }
    ]{re_gt_ci}{\RES}

    \begin{table}
        \caption{Summary statistics of all experiments.
            The experiments (Exp.), predictors (Pred.:Af = Affine; Gr = Greedy), and the algorithm (Alg.: GS = Gauss–Seidel 7 + 1d, Bg. = Background, new step lengths, Orig. = Background, original step lengths).
            Statistics marked with an asterisk$^*$ begin after the first 50 frames (Constant Motion) or 200 frames (other experiments).
            Mean, standard deviation (Std), and the 95\% confidence interval (CI) are for the relative function value error $J_{\operatorname{rel}}^k$, and the “GT” variants for the relative error to the ground truth, $e_{\operatorname{rel}}^k$, both defined in\eqref{eq:relmeasures}.
            The average real and CPU times spent processing each data frame, are in seconds.
        }
        \label{tab:summary}
        \setlength{\tabcolsep}{4pt}
        \centering
        \pgfplotstabletypeset[
            row predicate/.code={%
                    \pgfplotstablegetelem{#1}{predictor_name}\of\RES
                    \edef\target{NP}%
                    \ifx\pgfplotsretval\target
                        \pgfplotstableuserowfalse
                    \fi
                },
            every head row/.style={
                    before row={\toprule
                            & & & \multicolumn{3}{c}{Function value error$^*$} & \multicolumn{3}{c}{Ground truth error$^*$} & \multicolumn{2}{c}{Frame time}
                            \\
                            \cmidrule(lr){4-6}\cmidrule(lr){7-9}\cmidrule(lr){10-11}
                        },
                    after row=\midrule
                },
            every last row/.style={
                    after row=\bottomrule
                },
            every nth row={4}{before row=\midrule},
            columns={
                    experiment_name,
                    predictor_name,
                    algorithm_name,
                    re_meanN,
                    re_stdN,
                    re_ci,
                    re_gt_meanN,
                    re_gt_stdN,
                    re_gt_ci,
                    frame_time,
                    frame_cpu_time
                }
        ]{\RES}
    \end{table}

\fi
\endgroup

In \cref{fig:PDP1:gterror,fig:PDP2:gterror}, a comparison of the relative errors in the other experiments to those in the \emph{Constant Motion} experiment suggests that small levels of acceleration, with a changing direction, does not introduce significant additional errors.
In the \emph{Halting Motion} experiment, deceleration and acceleration is also present, as the inclusion slows down and stops at frames $1000$ and $2000$. This is followed by a delayed improvement in the relative error.

In the \emph{Disappearing Inclusions} experiment, the inclusions disappear abruptly at frames $500$ and $1000$ and reappear at frame $1500$, making the earlier reconstructions and predictions highly inaccurate.
We see this as spikes in the relative error, followed by a gradual decrease over subsequent iterations. \Cref{tab:summary} shows that the relative errors of the Gauss--Seidel $7+1$ method and of the background solver with the original step length adaption scheme, are comparable to those of the Exact method.
The background solver with the common step lengths parameters produces consistently higher relative errors.

To summarise the information shown in the graphs and tables, our proposed single-loop Gauss–Seidel approach is clearly superior to the alternatives.
It has a 5-6 times lower CPU footprint than the approach with an exact PDE solver, and shaves off a third of the CPU footprint of the old background solver of \cite{tuomov2024online-eit}.
While the latter might be an acceptable trade-off if the old approach were otherwise better, it, however, becomes very unstable in the Halting/Disappearing test scenarios well: the linearised model is not updated at a sufficient frequency compared to the nature of data evolution. More frequent updates could be imposed at a higher computational cost, comparable to our proposed single-loop solvers that have a simpler and more reliable computational architecture that does not depend on synchronisation between computational threads.
That being said, based on additional experiments, also the Gauss–Seidel approach becomes unstable if the number of forward steps is reduced from the 7 steps used here to 6 or less.
The background solver also heavily depends on the specific adaptive step length estimation scheme of \cite{tuomov2024online-eit}; it does not perform well with the simple fixed step sizes that we use with the other approaches.

\appendix

\section{Technical tracking estimates}
\label{sec:technical}

Following the structure of \cite[Appendix A]{tuomov2024tracking}, we prove here a simple scalar online tracking results, which are used to establish the results of \cref{sec:online-tracking,sec:inner-adjoint}.
The following assumption is a scalar variant of \cref{ass:tracking:main}, and differs from \cite[Assumption A.1]{tuomov2024online-eit} through the inclusion of the errors $\epsilon_k,\tilde\epsilon_k$.\begin{assumption}\label{ass:online-scalar-tracking:main}
    For a given $k \ge 0$ and scalars $d^u_{0}, \ldots, d^u_{k+1}, d^w_0, \ldots,  d^w_{k+1}, \varrho_{1}, \ldots, \varrho_{k} \ge 0$, and $\widetilde d_{0}, \ldots, \widetilde d_{k+1} \in \R$, there exist $\pi_u, \pi_w, \primaldifffact \alpha_w,\alpha_u>0$, and some errors $\epsilon_{k}, \tilde \epsilon_k \geq 0$,  such that%
    \begin{align}
        \label{item:online-scalar-tracking:main:inner-tracking}
        \tag{i}
        \kappa_u d^u_{j+1}
         &
        \le
        d^u_{j}
        + \pi_u(\varrho_{j} + \varepsilon_j)
         &   &
        \quad\text{for all}\quad j=1,\ldots,k,
        \\
        \label{item:online-scalar-tracking:main:adjoint-tracking}
        \tag{ii}
        \kappa_w d^w_{j+1}
         &
        \le
        d^w_j
                + {\primaldifffact} d^u_{j+1}
        + \pi_w(\varrho_{j} + \varepsilon_j)
         &   &
        \quad\text{for all}\quad j=1,\ldots,k,
        \quad\text{and}
        \\
        \label{item:online-scalar-tracking:main:differential-transformation}
        \tag{iii}
        \widetilde d_{j+1}
         &
        \le
        \alpha_u d^u_{j+1}
        + \alpha_w d^w_{j+1} + \tilde\epsilon_j
         &   &
        \quad\text{for all}\quad j=0,\ldots,k.
    \end{align}

\end{assumption}

Based on this recursive assumption, we aim to establish simple non-recursive bounds on $\widetilde d_{k+1}$.
Note that we could at the level of the assumption combine $\varrho_{j}$ and $\varepsilon_j$ into the single variable
\[
    \check \varrho_j \defeq \varrho_j + \epsilon_j,
\]
however, due to their different purposes, do not do this. This combination is, however, allows us use \cite{tuomov2024tracking} to deduce the next lemma that unrolls the recursion of \cref{item:online-scalar-tracking:main:inner-tracking,item:online-scalar-tracking:main:adjoint-tracking}.\begin{lemma}[{\cite[Lemma A.2]{tuomov2024tracking}}]
    \label{lemma:weaker:generic-recursion}
    Let \cref{ass:online-scalar-tracking:main}\,\cref{item:online-scalar-tracking:main:inner-tracking,item:online-scalar-tracking:main:adjoint-tracking} hold for a $k\geq 0$.
    Then, letting $\iota_k \defeq \sum_{m=1}^{k}  \kappa_u^{-m}\kappa_w^{-(k+1-m)}$ (understanding that $\iota_0=0$), we have
    \begin{equation}
        \label{eq:weaker:generic-recursion:claim}
        \begin{aligned}[t]
            \nexxt R
             &
            \defeq
            \alpha_u d_{k+1}^u + \alpha_w d_{k+1}^w
            \\
             &
            \le
            (\alpha_u\kappa_u^{-k} + \alpha_w \iota_k {\primaldifffact}) d_1^u
            +
            \alpha_w\kappa_w^{-k} d_1^w
            \\
            \MoveEqLeft[-1]
            + \sum_{j=0}^{k-1} \bigl(\alpha_u\kappa_u^{-(k-j)}\pi_u
            + \alpha_w[\iota_{k-j}{\primaldifffact}\pi_u + \kappa_w^{-(k-j)}\pi_w] \bigr) \check\varrho_{j+1}.
        \end{aligned}
    \end{equation}
\end{lemma}

The next two lemmas form our core estimates. They extend the corresponding lemmas of \cite{tuomov2024tracking} to $k$-dependent functions, with the corresponding additional errors from \cref{ass:online-scalar-tracking:main}.
To obtain the estimates, recalling that $\kappa_u, \kappa_w>1$, we observe that
\begin{equation}
    \label{eq:weaker:iota-est}
    p^k \iota_k
    \le
    \inv p k (\kappa/p)^{-(k+1)}
    \quad\text{for}\quad
    \kappa \defeq \min (\kappa_u, \kappa_w) > 1
    \text{ and any } p \in (0, \kappa).
\end{equation}
Thus, by sum formulae for arithmetic-geometric progressions \cite[formula 0.113]{gradshteyn2014table},
\begin{equation}
    \label{eq:weaker:iota-est-sum}
    \sum_{k=0}^{n-1} p^k \iota_k
    \le
    \sum_{k=0}^{\infty} p^k \iota_k
    \le
    p^{-1}(\kappa/p-1)^{-2}
    =p(\kappa-p)^{-2}
    \quad \text{for all } n \in \N.
\end{equation}\begin{lemma}\label{lemma:weaker:inner-product-error-estimate}
    Suppose \cref{ass:online-scalar-tracking:main} holds for a $k\geq 0$.
    Then for any $p\in(0,\kappa)$, we have
    \begin{equation}
        \label{eq:weaker:inner-product-error-estimate}
        \tilde d^2_{k+1}
        \le
        (\alpha_u d_{k+1}^u + \alpha_w d_{k+1}^w + \transformError{k})^2
        \le
        \check e_{p,k}.
    \end{equation}
    where, for
    $
        \trackingres_j
        \defeq \alpha_u\kappa_u^{-j}\pi_u + \alpha_w[\iota_j{\primaldifffact}\pi_u + \kappa_w^{-j}\pi_w],
    $
    and $\MAX\kappa \defeq \max\{\kappa_u,\kappa_w\}$,
    we set
    \begin{align}
        \label{eq:online-scalar-tracking:ressum}
        \trackingressum
         &
        \defeq
        \frac{\MAX\kappa}{p}
        \sum_{j=0}^\infty p^j \trackingres_j
        \le
        \frac{(\alpha_u\pi_u+\alpha_w\pi_w)\kappa\MAX\kappa}{p(\kappa-p)}
        + \frac{\alpha_w {\primaldifffact}\pi_u\MAX\kappa}{p^2(\kappa-p)^2}
        \quad\text{and}
        \\
        \label{eq:online-scalar-tracking:ek}
        \check e_{p,k}
         &
        \defeq
        \frac{5}{4}\left(
        \frac{\trackingressum(\alpha_u\kappa_u^{-k} + \alpha_w \iota_k {\primaldifffact})}{\pi_u p^k}
        (d_1^u)^2
        +
        \frac{\trackingressum\alpha_w\kappa_w^{-k}}{\pi_w p^k}
        (d_1^w)^2 + \sum_{j=0}^{k-1} \frac{\trackingressum\trackingres_{k-j}^2}{p^{k-j}}\check \varrho_{j+1}^2
        + 4 \transformError{k}^2
        \right).
    \end{align}
\end{lemma}\begin{proof}Invoking the inner and adjoint tracking \cref{ass:online-scalar-tracking:main}\,\cref{item:online-scalar-tracking:main:inner-tracking,item:online-scalar-tracking:main:adjoint-tracking} and \cref{lemma:weaker:generic-recursion}, we obtain
    \[
        \nexxt R
        \defeq
        \alpha_u d_{k+1}^u + \alpha_w d_{k+1}^w
        \le
        (\alpha_u\kappa_u^{-k} + \alpha_w \iota_k {\primaldifffact})d_{1}^u
        +
        \alpha_w\kappa_w^{-k} d_{1}^w
        + \sum_{j=0}^{k-1} \trackingres_{k-j} \check \varrho_{j+1}.
    \]
    Thus, using Young's inequality several times, we deduce for any $\theta_k^u,\theta_k^w,\theta_{k,j},s>0$ that
    \begin{equation}
        \label{eq:est0}
        \begin{aligned}[t]
            4s\nexxt R
             &
            \le
            \frac{(\alpha_u\kappa_u^{-k} + \alpha_w \iota_k {\primaldifffact})^2}{4\theta_k^u}
            (d_1^u)^2
            +
            \frac{(\alpha_w\kappa_w^{-k})^2}{4\theta_k^w}
            (d_1^w)^2
            \\
            \MoveEqLeft[-1]
            +
            \sum_{j=0}^{k-1} \frac{\trackingres_{k-j}^2 }{\theta_{k,j}}\check \varrho_{j+1}^2
            +
            4\left(\theta_k^u + \theta_k^w + \sum_{j=0}^{k-1} \theta_{k,j}\right)s^2.
        \end{aligned}
    \end{equation}
    Take $\theta_k^u =  p^k \trackingressum^{-1}\pi_u(\alpha_u\kappa_u^{-k} + \alpha_w \iota_k {\primaldifffact})$,
    $\theta_k^w =  p^k \trackingressum^{-1}\pi_w\alpha_w\kappa_w^{-k}$,
    and $\theta_{k,j} = \trackingressum^{-1}p^{k-j}\trackingres_{k-j}$.
    Observe that $\iota_k \leq \kappa_w \iota_{k+1}$ \cite[proof of Lemma A.2]{tuomov2024tracking}. Hence, we get $p^k\iota_k \leq (\kappa_w/p)p^{k+1}\iota_{k+1}$, and thus $p^k \psi_k \leq (\bar{\kappa}/p)p^{k+1}\psi_{k+1}$, where $\bar{\kappa}/p > 1$. Then we have
    \[
        \theta_k^u + \theta_k^w + \sum_{j=0}^{k-1} \theta_{k,j}
        = \frac{1}{\trackingressum}\left(p^k \trackingres_k + \sum_{j=1}^k p^j \trackingres_j\right)
        \le
        \frac{\bar{\kappa}}{p\trackingressum}\sum_{j=0}^{k+1} p^j \trackingres_j
        \le 1.
    \]
    By combining \eqref{eq:est0} with the previous bound and rearranging terms establishes
    \[
        4s\nexxt R - 4s^2
        \le
        \frac{\trackingressum(\alpha_u\kappa_u^{-k} + \alpha_w \iota_k {\primaldifffact})}{\pi_u p^k}
        (d_1^u)^2
        +
        \frac{\trackingressum\alpha_w\kappa_w^{-k}}{\pi_w p^k}
        (d_1^w)^2 + \sum_{j=0}^{k-1} \frac{\trackingressum\trackingres_{k-j}^2}{p^{k-j}}\check \varrho_{j+1}^2.
    \]
    We have $4s\transformError{k} \leq s^2 + 4\transformError{k}^2$,
    so it follows
    $
        4s(\nexxt R + \transformError{k}) - 5s^2
        \leq
        \frac{4}{5}\check e_{p,k}.
    $
    The left hand side is maximised by $s = \frac{2}{5}(\nexxt R + \transformError{k})$.
    This gives
    $
        (\nexxt R + \transformError{k})^2
        \leq
        \check e_{p,k}
    $
    By \cref{ass:online-scalar-tracking:main}\,\cref{item:online-scalar-tracking:main:differential-transformation}, $\widetilde d_{k+1} \leq \nexxt R + \transformError{k}$, so \eqref{eq:weaker:inner-product-error-estimate} follows.

    Finally, the bound in \eqref{eq:online-scalar-tracking:ressum} on $\trackingressum$ follows from \cref{eq:weaker:iota-est-sum} and $\sum_{j=0}^{\infty}(p/k)^j = 1/(1 - p/k) = \kappa/(\kappa - p)$.
\end{proof}

\begin{lemma}
    \label{lemma:weaker:error-sum}
    Suppose \cref{ass:online-scalar-tracking:main} holds for a $k\geq 0$. Then, for any $p\in[1,\kappa)$, we have
    \begin{equation}
        \label{eq:weaker:error-sum}
        {\widetilde d_{k+1}}^2 \leq \frac{5}{4}\trackingressum^2 \check\varrho_{k+1}^2 + e_{p,k},
    \end{equation}
    where, for $\check e_{p,k}$ defined in \cref{eq:online-scalar-tracking:ek},
    \begin{equation}
        e_{p,k}\defeq \check e_{p,k} - \frac{5}{4}\trackingressum^2 \check\varrho_{k+1}^2
    \end{equation}
    satisfies
    \[
        \begin{split}
            \sum_{n=0}^{k} p^n e_{p,n}
             &
            \le
            \Psi_p
            \defeq
            \frac{5}{4}\left(
            \frac{(d_1^u)^2}{\pi_u} \bigg(\frac{\trackingressum\alpha_u\kappa}{\kappa-1} + \frac{\trackingressum\alpha_w{\primaldifffact}}{(\kappa-1)^2}\bigg)
            +
            \frac{(d_1^w)^2}{\pi_w} \bigg(\frac{\trackingressum\alpha_w\kappa}{\kappa-1}\bigg)
            +
            \sum_{n=0}^{k} p^n \transformError{n}^2
            \right).
        \end{split}
    \]
\end{lemma}\begin{proof}The proof is analogous to that of \cite[Lemma A.4]{tuomov2024tracking}.
\end{proof}

Finally we have the following result similar to \cite[Lemma A.5]{tuomov2024tracking}.\begin{lemma}Suppose \cref{ass:online-scalar-tracking:main} holds for a $k \in \N$.
    Then, for $\breve e_{1,n}$ given in \eqref{eq:online-scalar-tracking:ek}, we have
    $
        \sum_{n=0}^{k-1} \breve e_{1,n}
        \le \Psi_1 + \frac{5}{4}\trackingressum[1]^2\sum_{n=0}^{k-1}\check\varrho_{n+1}^2.
    $
\end{lemma}

\begin{proof}We have $\breve e_{1,n} = e_{1,n} + \frac{5}{4}\trackingressum[1]^2 \check\varrho_{n+1}^2$,
    where \eqref{eq:online-scalar-tracking:ressum} bounds $\sum_{n=0}^{k-1} e_{1,n} \le \Psi_1$.
\end{proof}

\bibliographystyle{jnsao}
\end{document}